\documentclass{article}

\usepackage{amsmath,amsthm,amssymb,enumerate}
\usepackage[all]{xy}
\usepackage{rotating}
\usepackage{tikz}
\usetikzlibrary{arrows,automata}

\usepackage{algorithm}
\usepackage[noend]{algorithmic}

\def\llex{<_\mathrm{lex}}
\def\lelex{\le_\mathrm{lex}}
\def\A{{\mathcal A}}
\def\N{{\mathbb N}}
\def\SS{{\mathfrak S}}
\newcommand{\Sym}{\Sigma}
\newcommand{\Magma}{\textsc{Magma}}
\newcommand{\claim}[1]{Claim~(\emph{#1}\hspace*{1pt})}
\newcommand{\claims}[1]{Claims~(\emph{#1}\hspace*{1pt})}
\newcommand{\noclaim}[1]{(\emph{#1}\hspace*{1pt})}

\newtheorem{theorem}{Theorem}[section]
\newtheorem{lemma}[theorem]{Lemma}
\newtheorem{proposition}[theorem]{Proposition}
\newtheorem{corollary}[theorem]{Corollary}
\newtheorem{remark}[theorem]{Remark}
\newtheorem{definition}[theorem]{Definition}
\newtheorem{example}[theorem]{Example}

\title{Generating random braids}
\author{Volker Gebhardt\footnote{Both authors acknowledge support under Australian Research Council's Discovery Projects funding scheme (project number DP1094072), and the Spanish Project MTM2010-19355.}, Juan Gonz\'alez-Meneses\footnote{Partially supported by Project P09-FQM-5112 and FEDER.}}
\date{June 30, 2012}

\begin{document}

\maketitle


\begin{abstract}
We present an algorithm to generate positive braids of a given length as words in Artin generators with a uniform probability. The complexity of this algorithm is polynomial in the number of strands and in the length of the generated braids.

As a byproduct, we describe a finite state automaton accepting the language of lexi\-cographically minimal representatives of positive braids that has the minimal possible number of states,
and we prove that its number of states is exponential in the number of strands.
\end{abstract}

\section{Introduction}\label{S:introduction}

Recently, many research papers have appeared that describe effective computations with braids on $n$ strands.
In many cases, the results of performing a certain computation for a set of `random' or `pseudo-random' braids of a given length are reported.
Usually, the authors generate so-called {\it positive braids}, that is, elements of the submonoid $B_n^+$ of the braid group $B_n$ generated (as a monoid) by the standard Artin generators $\{\sigma_1,\ldots,\sigma_{n-1}\}$.
Methods commonly used to generate such positive braids are the following:

\begin{enumerate}\vspace{-1ex}\addtolength{\itemsep}{-1ex}
\item
In order to generate a positive braid of length $k$, choose $k$ times an element of $\{\sigma_1,\ldots,\sigma_{n-1}\}$ with uniform probability, and form their product.
\item
In order to generate a positive braid of canonical length $k$, choose $k$ times a simple braid with uniform probability on the (finite) set of simple braids and compute their product; if the obtained braid has canonical length smaller than $k$, discard it and try again.
\end{enumerate}\vspace{-1ex}
However, none of these procedures generates braids with a uniform distribution; some braids are far more likely to appear than others.
For instance, if one uses the first procedure to produce positive braids of length 6 in $B_4$, the probability of obtaining the braid $(\sigma_1)^6$ is $3^{-6}$, as there is only one way to write this braid as a positive word in Artin generators.  On the other hand, the probability of obtaining the braid $\Delta=\sigma_1\sigma_2\sigma_1\sigma_3\sigma_2\sigma_1$ is $16\cdot3^{-6}$, as there are 16 distinct ways to write $\Delta\in B_4$ as a positive word in Artin generators~\cite{Stanley}.
That is, the above procedure is 16 times more likely to generate $\Delta$ than to generate $(\sigma_1)^6$. This bias becomes more dramatic as $n$ and the length of the braids involved increase.

In this paper we shall give a procedure to generate random positive braids of given length with a uniform probability.
That is, given $n$ and $k$ as input, the algorithm generates a positive braid in $B_n$ of length $k$, in such a way that the probability of obtaining any given braid is $1/x_{n,k}$, where $x_{n,k}$ is the number of positive braids of length $k$ in $B_n$.

The structure of the paper is as follows.
Section~\ref{S:overview_algorithm} describes the basic idea of generating uniformly random positive braids via lexicographically minimal representative words (lex-representatives, for short).
In Section~\ref{S:counting_braids}, we count the braids in $B_n^+$ of a given length~$k$ using a result by Bronfman.
In Section~\ref{S:Forbidden_prefixes}, we develop a description of lex-representatives, which will be used in Section~\ref{S:counting_braids_with_prefixes} to count the lex-representatives that start with a given prefix, completing the description of our algorithm.

It is known that the set of lex-representatives is a regular language.
In Section~\ref{S:automaton}, we show that our description of lex-representatives yields an acceptor for this regular language that has the minimal number of states, and that the number of states is exponential in~$n$.  This shows, in particular, that using standard language theoretical techniques to generate uniformly random positive braids is not efficient.
Finally, in Section~\ref{S:Complexity}, we analyse the complexity of our algorithm and give timing results.

{\bf Acknowledgements:} We thank Pascal Weil, Fr\'ed\'erique Bassino and Cyril Nicaud, for suggesting to us the use of automatic structures to generate random elements, which led us to realise that our sets of {\it minimal forbidden prefixes} provide a minimal finite state automaton.

\section{Structure of the algorithm}\label{S:overview_algorithm}

The basic idea of our algorithm to produce random positive braids is the following. Let $\A_n=\{\sigma_1,\ldots,\sigma_{n-1}\}$ be the set of standard generators (or {\it atoms}) in $B_n^+$. Let $\A_n^*$ be the free monoid generated by $\A_n$. We know there is a morphism of monoids $b:\: \A_n^* \rightarrow B_n^+$ which sends each element in $\A_n^*$ (a word in $\sigma_1,\ldots,\sigma_{n-1}$) to the positive braid it represents. As we saw above, the map $b$ is not injective, so we are going to define a
section of $b$. For that purpose, we will order the elements in $\A_n^*$ having the same length using $\llex$, which is the lexicographical order in which $\sigma_1<\sigma_2<\cdots<\sigma_{n-1}$.

\begin{definition}
Given $\beta\in B_n^+$, we define the {\bf lex-representative} of $\beta$ to be
$$
  \omega(\beta)=\min_{\llex}\{b^{-1}(\beta)\} \in {\A_n^*}.
$$
\end{definition}

In other words, $\omega(\beta)$ is the smallest positive word, with respect to $\llex$, that represents $\beta$. Notice that this is well defined as $B_n^+$ is a homogeneous monoid (all words representing a given element have the same length). For instance $\omega(\sigma_3\sigma_1)= \sigma_1\sigma_3$, $\omega(\sigma_2\sigma_1\sigma_2)=\sigma_1\sigma_2\sigma_1$ and for $\Delta\in B_4$, $\omega(\Delta)=\sigma_1\sigma_2\sigma_1\sigma_3\sigma_2\sigma_1$. Here the arguments of $\omega$ are braids in $B_n^+$, whereas the images are words in $\A_n^*$.

It is clear that $\omega:\: B_n^+\rightarrow \A_n^*$ is a
section of $b$, so it is injective. Moreover, as $B_n^+$ is homogeneous, we can define $(B_n^+)_k$ to be the set of positive braids of length $k$, and $L_{n,k}$ to be the set of lex-representatives of length $k$ ($L$ stands for language). Then $\omega: \: (B_n^+)_k \rightarrow L_{n,k}$ is a bijection, whose inverse is $b_{|_{L_{n,k}}}$.

Our algorithm will produce a random element of $L_{n,k}$ with uniform probability.  By the above arguments, this is equivalent to producing a random positive braid of length $k$ with uniform probability on $(B_n^+)_k$. There are three main steps:
\begin{enumerate}\vspace{-1ex}\addtolength{\itemsep}{-1ex}
\item Determine the size of $L_{n,k}$, that is, compute $x_{n,k}=\big|(B_n^+)_k\big|$.
\item Choose a random integer $r$ between 1 and $x_{n,k}$.
\item Find the $r$-th word in $L_{n,k}$, where $L_{n,k}$ is ordered by $\llex$.
\end{enumerate}\vspace{-1ex}

The second step poses no difficulty (assuming we know how to generate random integers), so our task is to be able to perform steps 1 and 3 in polynomial time.

It is useful to identify the lex-representatives of braids in $B_n^+$ as the vertices of a rooted tree, with the elements of $L_{n,k}$ being the vertices at depth $k$ and edges given by the prefix partial order in the monoid $\A_n^*$.  (That is, the root corresponds to the trivial word $\epsilon$, and given a word $w\in L_{n,s}$ and $\sigma_i\in\A_n$ such that $w\sigma_i \in L_{n,s+1}$, there is an edge labelled $\sigma_i$ joining $w$ to $w\sigma_i$.)

Figure~\ref{F:rooted_tree} shows the tree of lex-representatives in $B_4^+$ truncated at depth $3$; the elements of $L_{4,3}$ correspond to the leaves of the truncated tree.
We have $|L_{4,0}|=1$, $|L_{4,1}|=3$, $|L_{4,2}|=8$ and $|L_{4,3}|=19$. Notice that there are no vertices corresponding to the words $\sigma_3\sigma_1$ and $\sigma_2\sigma_1\sigma_2$, as these words are not lex-representatives.
The fact that there are 19 leaves means that there are 19 positive braids in $B_4^+$ of length~$3$.
Hence, in this example, our algorithm will choose a random number $r$ between 1 and 19, and it will look for the $r$\nobreakdash-th leaf of the rooted tree, counting from left to right, as leaves are ordered lexicographically by construction.

\begin{figure}[ht]
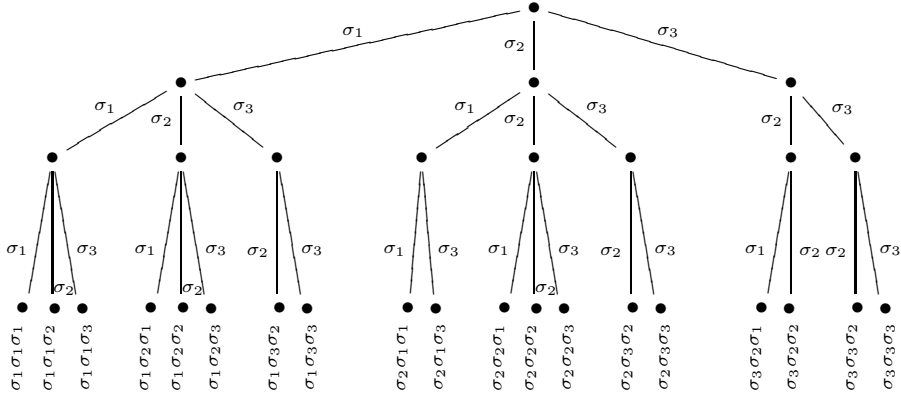

\[\xygraph{
!{<0cm,0cm>;<0.427cm,0cm>:<0cm,1cm>::}
!{(16,3) }*+{\bullet}="1"
!{(5,2) }*+{\bullet}="a"
!{(16,2) }*+{\bullet}="b"
!{(24,2) }*+{\bullet}="c"
!{(1,1) }*+{\bullet}="aa"
!{(5,1) }*+{\bullet}="ab"
!{(8,1) }*+{\bullet}="ac"
!{(12.5,1) }*+{\bullet}="ba"
!{(16,1) }*+{\bullet}="bb"
!{(19,1) }*+{\bullet}="bc"
!{(24,1) }*+{\bullet}="cb"
!{(26,1) }*+{\bullet}="cc"
!{(0,-1.5) }*+{\begin{sideways}$^{\scriptstyle\sigma_1\sigma_1\sigma_1}\;\bullet$\end{sideways}}="aaa"
!{(1,-1.5) }*+{\begin{sideways}$^{\scriptstyle\sigma_1\sigma_1\sigma_2}\;\bullet$\end{sideways}}="aab"
!{(2,-1.5) }*+{\begin{sideways}${\scriptstyle\sigma_1\sigma_1\sigma_3}\;\bullet$\end{sideways}}="aac"
!{(4,-1.5) }*+{\begin{sideways}$^{\scriptstyle\sigma_1\sigma_2\sigma_1}\;\bullet$\end{sideways}}="aba"
!{(5,-1.5) }*+{\begin{sideways}$^{\scriptstyle\sigma_1\sigma_2\sigma_2}\;\bullet$\end{sideways}}="abb"
!{(6,-1.5) }*+{\begin{sideways}${\scriptstyle\sigma_1\sigma_2\sigma_3}\;\bullet$\end{sideways}}="abc"
!{(8,-1.5) }*+{\begin{sideways}$^{\scriptstyle\sigma_1\sigma_3\sigma_2}\;\bullet$\end{sideways}}="acb"
!{(9,-1.5) }*+{\begin{sideways}${\scriptstyle\sigma_1\sigma_3\sigma_3}\;\bullet$\end{sideways}}="acc"
!{(12,-1.5) }*+{\begin{sideways}$^{\scriptstyle\sigma_2\sigma_1\sigma_1}\;\bullet$\end{sideways}}="baa"
!{(13,-1.5) }*+{\begin{sideways}${\scriptstyle\sigma_2\sigma_1\sigma_3}\;\bullet$\end{sideways}}="bac"
!{(15,-1.5) }*+{\begin{sideways}$^{\scriptstyle\sigma_2\sigma_2\sigma_1}\;\bullet$\end{sideways}}="bba"
!{(16,-1.5) }*+{\begin{sideways}$^{\scriptstyle\sigma_2\sigma_2\sigma_2}\;\bullet$\end{sideways}}="bbb"
!{(17,-1.5) }*+{\begin{sideways}${\scriptstyle\sigma_2\sigma_2\sigma_3}\;\bullet$\end{sideways}}="bbc"
!{(19,-1.5) }*+{\begin{sideways}$^{\scriptstyle\sigma_2\sigma_3\sigma_2}\;\bullet$\end{sideways}}="bcb"
!{(20,-1.5) }*+{\begin{sideways}${\scriptstyle\sigma_2\sigma_3\sigma_3}\;\bullet$\end{sideways}}="bcc"
!{(23,-1.5) }*+{\begin{sideways}$^{\scriptstyle\sigma_3\sigma_2\sigma_1}\;\bullet$\end{sideways}}="cba"
!{(24,-1.5) }*+{\begin{sideways}${\scriptstyle\sigma_3\sigma_2\sigma_2}\;\bullet$\end{sideways}}="cbb"
!{(26,-1.5) }*+{\begin{sideways}$^{\scriptstyle\sigma_3\sigma_3\sigma_2}\;\bullet$\end{sideways}}="ccb"
!{(27,-1.5) }*+{\begin{sideways}${\scriptstyle\sigma_3\sigma_3\sigma_3}\;\bullet$\end{sideways}}="ccc"
"1"-"a"_{\sigma_1}
"1"-"b"_{\sigma_2}
"1"-"c"^{\sigma_3}
"a"-"aa"_{\sigma_1}
"a"-"ab"_{\sigma_2}
"a"-"ac"^{\sigma_3}
"b"-"ba"_{\sigma_1}
"b"-"bb"_{\sigma_2}
"b"-"bc"^{\sigma_3}
"c"-"cb"_{\sigma_2}
"c"-"cc"^{\sigma_3}
"aa"-"aaa"_{\sigma_1}
"aa"-"aab"^(0.7){\!\!\sigma_2}
"aa"-"aac"^{\sigma_3}
"ab"-"aba"_{\sigma_1}
"ab"-"abb"^(0.7){\!\!\sigma_2}
"ab"-"abc"^{\sigma_3}
"ac"-"acb"_{\sigma_2}
"ac"-"acc"^{\sigma_3}
"ba"-"baa"_{\sigma_1}
"ba"-"bac"^{\sigma_3}
"bb"-"bba"_{\sigma_1}
"bb"-"bbb"^(0.7){\!\!\sigma_2}
"bb"-"bbc"^{\sigma_3}
"bc"-"bcb"_{\sigma_2}
"bc"-"bcc"^{\sigma_3}
"cb"-"cba"_{\sigma_1}
"cb"-"cbb"^{\sigma_2}
"cc"-"ccb"_{\sigma_2}
"cc"-"ccc"^{\sigma_3}
}\]\vskip-3ex
\caption{The rooted tree corresponding to $L_{4,3}$}
\label{F:rooted_tree}
\end{figure}

Of course, it is not efficient at all to try to compute the whole set $L_{n,k}$, as it has exponential size with respect to $k$. We shall overcome this problem by counting suitable subsets of leaves of the truncated rooted tree, as follows.

\begin{definition}
\label{D:x_n_k}
Given a word $w\in \A_n^*$ and an integer $m\in \{0,\ldots,n-1\}$, we define $x_{n,k}(w,m)$ to be the number of words in $L_{n,k}$ of the form $w w'$, where $w'\in \A_n^*$ does not start with $\sigma_1,\ldots,\sigma_m$.
\end{definition}

Obviously, $x_{n,k}(w,m) = 0$ if $w$ is not a lex-representative, or if $|w|>k$.
Moreover, if $w$ is a lex-representative, then $x_{n,k}(w,n-1) = 1$ if $k=|w|$ and $x_{n,k}(w,n-1) = 0$ otherwise.
We will show in Section~\ref{S:counting_braids} how to compute $x_{n,k}(\epsilon,0)=x_{n,k}$, and in Section~\ref{S:counting_braids_with_prefixes} how to compute $x_{n,k}(\epsilon,m)$ for $m\ge1$, and also $x_{n,k}(w,m)$ where $w=w'\sigma_j$ and $m\geq j-1$, in polynomial time and space with respect to $n$ and $k$. If we assume this to be known, we can explain our main algorithm with some more detail:

\begin{theorem}\label{T:Algorithm}
Suppose the time and space required to compute $x_{n,k}(w,m)$, where $m\ge0$ if $w=\epsilon$ and $m\ge j-1$ if $w=w'\sigma_j$, is polynomial in $n$ and~$k$.
Then there is an algorithm generating a random element of $\big(B_n^+\big)_k$, with uniform probability, whose time and space complexity is polynomial in $n$ and $k$.
\end{theorem}

\begin{proof}
In order to generate a random element of $\big(B_n^+\big)_k$, we compute $x_{n,k}$ and choose a random integer $r$ between $1$ and $x_{n,k}$.
It just remains to determine the $r$-th element $w_{(r)}$ of the chain $L_{n,k}$ with respect to $\llex$.
We will find the word $w_{(r)}$ letter by letter.  (Pseudocode is given in Algorithm~\ref{A:random}.)

Suppose that we have computed a prefix $w$ of $w_{(r)}$ and know that there are $\nu$ words in $L_{n,k}$ of the form $w w'$, where $w'\in \A_n^*$ and $w_{(r)}\llex w w'$.  (Initially, $w=\epsilon$ and $\nu = x_{n,k}-r$.)
If $w\ne\epsilon$, let $\sigma_j$ be the last letter of $w$; otherwise let $j=0$.
Let $\sigma_i$ be the next letter of $w_{(r)}$.
Observe that $i\ge j-1$, since otherwise $w_{(r)}$ would not be lexicographically minimal.
By Definition~\ref{D:x_n_k}, we have $m\ge i$ if and only if $x_{n,k}(w,m) \le \nu$.
Hence, we have $i = \min S$, where $S = \big\{ m\in\{j-1,\dots,n-1\} \mid x_{n,k}(w,m) \le \nu \big\} \ni n-1$.
We can determine $i$ using binary search in at most $\lceil\log_2(n-1)\rceil$ steps, each step requiring the computation of one number of the form $x_{n,k}(w,m)$.
We then replace $w$ by $w\sigma_i$ and $\nu$ by $\nu-x_{n,k}(w,i)$, and proceed to the next letter of $w_{(r)}$.
(While this does not affect the complexity, we remark that $x_{n,k}(w,i)$ is known from the binary search, so updating $\nu$ just requires one subtraction.)

As $w_{(r)}$ has length $k$, at most $k\lceil\log_2(n-1)\rceil$ steps are required to find $w_{(r)}$; since each step is polynomial in $n$ and $k$, the claim follows.
\end{proof}

\begin{example}
Suppose we want to generate a random braid of length 3 in $B_4^+$ (see Figure~\ref{F:rooted_tree}). Compute $x_{4,3}=19$. Choose a random number $r$ between 1 and 19. Say $r=16$, so $\nu=3$. Now compute $x_{4,3}(\epsilon,2)=4>3=\nu$. Thus, the first letter of $w_{(16)}$ is $\sigma_3$.
Since $x_{4,3}(\epsilon,3)=0$, the value of $\nu$ remains unchanged.

Now we find the second letter of $w_{(16)}$, which must belong to the set $\{\sigma_2,\sigma_3\}$ (as the first letter is $\sigma_3$).
We compute $x_{4,3}(\sigma_3,2)=2\le3=\nu$.  Hence, the second letter of $w_{(16)}$ is $\sigma_2$, and we must replace $\nu$ by $\nu-x_{4,3}(\sigma_3,2)=1$.

Finally, as the second letter of $w_{(16)}$ is $\sigma_2$, the third one could be any letter in $\{\sigma_1,\sigma_2,\sigma_3\}$. We compute  $x_{4,3}(\sigma_3\sigma_2,2)=0\le1=\nu$,
then $x_{4,3}(\sigma_3\sigma_2,1)=1\le1=\nu$, and conclude that the third letter of $w_{(16)}$ is $\sigma_1$.  Thus, $w_{(16)}= \sigma_3\sigma_2\sigma_1$.
\end{example}

Algorithm~\ref{A:random} gives pseudocode for finding $w_{(r)}$, assuming we know how to compute $x_{n,k}(w,m)$ as in the statement of Theorem~\ref{T:Algorithm}.  It contains references to two algorithms which will be introduced in Section~\ref{S:counting_braids} respectively Section~\ref{S:counting_braids_with_prefixes}.
In Section~\ref{S:counting_braids} we will use a known result~\cite{Bronfman} to compute $x_{n,k}(\epsilon,0)=x_{n,k}$ (line~\ref{A:random:Bronfman} of Algorithm~\ref{A:random}).
An algorithm for computing all other required instances of $x_{n,k}(w,m)$ (line~\ref{A:random:count_leaves} of Algorithm~\ref{A:random}) will be given in Section~\ref{S:counting_braids_with_prefixes}.

\begin{algorithm}[ht]
\caption{Producing a uniformly random braid in $B_n^+$ of length $k$}
\label{A:random}
\small\begin{algorithmic}[1]
\REQUIRE{Integers $n\ge2$ and $k\ge 0$.}
\ENSURE{A braid in $B_n^+$ of length $k$.}\smallskip
\STATE{Compute $x_{n,k}$. [Section~\ref{S:counting_braids};
                           equations (\ref{E:x_n_j}) and (\ref{E:h_m_j})]}\label{A:random:Bronfman}
\STATE{Choose a random integer $r\in\{1,\ldots, x_{n,k}\}$ with uniform probability.}\label{A:random:random}
\STATE{$w := \varepsilon$\,;\, $\nu := x_{n,k}-r$\,;\, $a := 0$}
\FOR{$l := 1$ to $k$}
  \STATE{$a := \max\{a-1,1\}$\,;\, $b := n-1$\,;\, $\mu := 0$\quad/* $\mu = x_{n,k}(w,b)$ at all times */}%
                                                                                \label{A:random:for-start}
  \WHILE{$a<b$}
    \STATE{$m := \lfloor\frac{a+b}{2} \rfloor$}\label{A:random:while-start}
    \STATE{Compute $x_{n,k}(w,m)$. [Algorithm~\ref{A:countleaves}]}\label{A:random:count_leaves}
    \IF{$x_{n,k}(w,m) \le \nu$}
      \STATE{$b := m$\,;\, $\mu := x_{n,k}(w,m)$}
    \ELSE
      \STATE{$a := m+1$}
    \ENDIF
  \ENDWHILE\label{A:random:while-end}
  \STATE{$w := w\sigma_a$\,;\, $\nu := \nu - \mu$\quad/* at this point $a=b$, hence $\mu = x_{n,k}(w,a)$ */}
\ENDFOR\label{A:random:for-end}
\RETURN{$b(w)$}
\end{algorithmic}\vskip0.5ex
\end{algorithm}

\section{Counting all positive braids of a given length}\label{S:counting_braids}

In this section we use a formula given by Bronfman~\cite{Bronfman} to describe how to compute the number of elements of $\big(B_n^+\big)_k$ in time respectively space that is polynomial in $n$ and $k$. To our knowledge, the paper~\cite{Bronfman} was never published, but one can find proofs of its main result in~\cite{Albenque-Nadeau} and also in \cite{GM2011}. Bronfman gave a recurrence relation for the growth function of the monoid $B_n^+$. More precisely, if $x_{n,k}$ is the number of elements of length $k$ in $B_n^+$, then the formal power series
$$
     G_n(t)=\sum_{k\geq 0}{x_{n,k}\, t^k}
$$
is the growth function of $B_n^+$.
Deligne~\cite{Deligne} showed that this function was rational, namely the inverse of a polynomial $H_n(t)$. The following recurrence relation to compute this polynomial was given in~\cite{Bronfman}:
\begin{equation}\label{E:Bronfman}
    H_n(t)=\sum_{i=1}^n{(-1)^{i+1}t^{\binom{i}{2}}H_{n-i}(t)}\,,
\end{equation}
where $H_0(t)=H_1(t)=1$.
In particular, $H_n(t)$ is a polynomial of degree $\binom{n}{2}$ such that $H_n(0)=1$.
Denoting the coefficient of $t^m$ in $H_i(t)$ by $h_{i,m}$, it follows from $G_n(t)H_n(t)=1$ that for any $j>0$ one has
\begin{equation}\label{E:x_n_j}
    x_{n,j} =  -\big( x_{n,j-1}h_{n,1} + x_{n,j-2}h_{n,2} + \cdots + x_{n,0}h_{n,j} \big)\;,
\end{equation}
where the sum has at most $\min\{j,\binom{n}{2}\}$ terms, since $h_{n,m}=0$ if $m>\binom{n}{2}$.
Therefore, we can iteratively compute the numbers $x_{n,0},x_{n,1},\ldots,x_{n,k}$, knowing the coefficients $h_{n,1},\ldots,h_{n,k}$.
From~(\ref{E:Bronfman}), the latter can in turn be calculated with the recurrence relation
\begin{equation}\label{E:h_m_j}
   h_{m,j}=\sum_{i=1}^{m}(-1)^{i+1} h_{m-i,j-\binom{i}{2}} \;,
\end{equation}
where we define $h_{m-i,j-\binom{i}{2}}=0$ if $j-\binom{i}{2}<0$.  Note that in order to compute $x_{n,k}$, only the coefficients $h_{m,j}$ for $0\le m\le n$ and $0\le j\le k$ are required.  In particular, time and space required to compute $x_{n,k}$ are polynomial in $k$ and $n$.%
\smallskip

The coefficients $h_{m,j}$ as well as the integers $x_{m,j}$ for $0\le m\le n$ and $0\le j\le k$  will be used later in our algorithm generating random elements of $B_n^+$.
Note that coefficients $h_{m,j}$ that have already been computed do not change if $n$ and $k$ are changed; one merely might have to compute additional coefficients.
Thus, any values of $h_{m,j}$ and $x_{m,j}$ that are computed can be stored once and for all; their computation should be thought of as a precomputation.

\section{Forbidden prefixes}\label{S:Forbidden_prefixes}

It only remains to compute $x_{n,k}(w,m)$, where $m\geq 1$ if $w=\epsilon$, and $m\geq i-1$ if $w=w'\sigma_i$. (Recall Definition~\ref{D:x_n_k}.)
This calculation, which will be done in Section~\ref{S:counting_braids_with_prefixes}, relies on a detailed description of the words $w'$ that can be appended to $w$ so that the concatenation $ww'$ is still a lex-representative in terms of what we call {\it forbidden prefixes}.

Denote by
$$
 L_n = \bigcup_{k\geq 0}L_{n,k} = \bigcup_{k\geq 0}\omega\big((B_n^+)_k\big) \subseteq \A_n^*
$$
the set of words in $\{\sigma_1,\ldots,\sigma_{n-1}\}$ that are lex-representatives.
To simplify notation, let $M=B_n^+$ and $M_k=(B_n^+)_k$ for $k\in\N$.
We denote the length of a word $w\in\A_n^*$ by $|w|$,
and the length of a braid $x\in M$ by $|x|$.

\begin{definition}
For $w\in L_{n}$, we define the set of {\bf forbidden prefixes after~$w$} as
$F_n(w)=\{\alpha\in M\mid w \omega(\alpha) \notin  L_n \}$.
\end{definition}

For a braid $\alpha\in M$, we define the set of {\it multiples} of $\alpha$ as
$\alpha M = \{\alpha \beta \mid \beta\in M\}$.
The following is an important property of the sets of forbidden prefixes.

\begin{lemma}\label{L:quadrants}
Let $w\in L_n$. If $\alpha\in F_n(w)$, then $\alpha M\subseteq F_n(w)$.
\end{lemma}

\begin{proof}
For $u\in\A_n^*$ and $k\in\N$, let $u_{|_k}$ denote the initial subword of length $k$ of~$u$.
As $\alpha\in F_n(w)$, we have $\omega(b(w)\alpha)\llex w\omega(\alpha)$, and thus
$\omega(b(w)\alpha)_{|_{|w|}}\llex w$.
For any $\beta\in M$, the latter implies
$\omega(b(w)\alpha\beta) \lelex \omega(b(w)\alpha)\omega(\beta) \llex w\omega(\alpha\beta)$,
proving the claim.
\end{proof}

As $M$ is a cancellative monoid, it admits a well-defined partial order $\preccurlyeq$, where $a\preccurlyeq b$ if there exists $c\in M$ such that $ac=b$; in this case we say that $a$ is a {\it prefix} of $b$.
It is well known that $(M,\preccurlyeq)$ is Noetherian, that is, that there are no infinite descending chains with respect to $\preccurlyeq$ in $M$.
We can then define a special subset of the forbidden prefixes.

\begin{definition}\label{D:forbidden_prefixes}
For $w\in L_n$, the set of {\bf minimal forbidden prefixes after~$w$}, denoted $F_n^{\min}(w)$, is the set of minimal elements, with respect to $\preccurlyeq$, in $F_n(w)$.
\end{definition}

From Lemma~\ref{L:quadrants} and the Noetherianity of $M$, one has
$F_n(w)=\bigcup_{\alpha\in F_n^{\min}(w)}\alpha M$.

The rest of this section will be devoted to the study of $F_n^{\min}(w)$.
We will show that it is a finite set,
we will describe its elements,
and we will show how to compute them in linear time with respect to $n$ and $|w|$.

First we recall a particularly useful property of the prefix order $\preccurlyeq$ of $M$: It is a lattice order \cite{Epstein},
that is, for any $a,b\in M$ there are a greatest common prefix $a\wedge b$ and a least common multiple $a\vee b$ with respect to $\preccurlyeq$. In the particular case of Artin generators, we have $\sigma_i\vee \sigma_j=\sigma_i\sigma_j = \sigma_j\sigma_i$ if $|i-j|>1$, and $\sigma_i\vee \sigma_j = \sigma_i\sigma_j\sigma_i = \sigma_j\sigma_i\sigma_j $ if $|i-j|=1$.

Given two elements $a,b\in M$, there are positive elements $a\backslash b$ and $b\backslash a$, such that $a\vee b = a(a\backslash b) = b (b\backslash a)$. As the monoid $M$ embeds in the group $B_n$, we can write $a\backslash b= a^{-1}(a\vee b)$ and  $b\backslash a = b^{-1}(a\vee b)$. Notice that $\sigma_j\backslash \sigma_i=\sigma_i$ if $|i-j|>1$ and $\sigma_j\backslash \sigma_i = \sigma_i\sigma_j$ if $|i-j|=1$.

We shall frequently use the following property:

\begin{lemma}
If $a,b,c\in M$, then $a\preccurlyeq bc$ is equivalent to $b\backslash a \preccurlyeq c$.
\end{lemma}

\begin{proof}
As $b\preccurlyeq bc$, it follows that $a\preccurlyeq bc$ if and only if $a\vee b \preccurlyeq bc$;
the latter is equivalent to $b^{-1}(a\vee b) \preccurlyeq c$, so the claim is shown.
\end{proof}

The following two results form the basis for an inductive description of $F_n^{\min}(w)$.

\begin{proposition}\label{P:minimal_generator}
\quad\vspace*{-1ex}
\begin{enumerate}[\textup{(}a\textup{)}]
\item\label{P:minimal_generator:a}
$F_n(\epsilon)=F_n^{\min}(\epsilon)=\emptyset$.
\item\label{P:minimal_generator:b}
$F_n^{\min}(\sigma_j)=\{\sigma_1,\ldots,\sigma_{j-2},\sigma_{j-1}\sigma_j\}$.

(Here and in the sequel, we use the convention that words involving indices less than 1 are ignored. That is, $F_n^{\min}(\sigma_1)=\emptyset$ and $F_n^{\min}(\sigma_2)=\{\sigma_1\sigma_2\}$.)
\end{enumerate}
\end{proposition}

\begin{proof}
It is obvious that
$\{\sigma_1,\ldots,\sigma_{j-2},\sigma_{j-1}\sigma_j\} \subseteq F_n(\sigma_j)$.  Indeed, $\sigma_1,\ldots,\sigma_{j-2}$ are atoms and $\sigma_{j-1}\notin F_n(\sigma_j)$,
so $\{\sigma_1,\ldots,\sigma_{j-2},\sigma_{j-1}\sigma_j\} \subseteq F_n^{\min}(\sigma_j)$.

Conversely, if $\alpha\in F_n(\sigma_j)$ then $\omega(\sigma_j \alpha)\llex \sigma_j \omega(\alpha)$, and hence $\omega(\sigma_j \alpha)_{|_1}\llex \sigma_j$.
This means that $\sigma_i\preccurlyeq \sigma_j\alpha$ for some $i<j$,
whence $\alpha$ admits a prefix from the set
$ \{\sigma_j\backslash \sigma_i \mid i=1,\ldots,j-1\}
  = \{\sigma_1,\ldots,\sigma_{j-2},\sigma_{j-1}\sigma_j\}$.

This proves \claim{\ref{P:minimal_generator:b}}.  \claim{\ref{P:minimal_generator:a}} is trivial.
\end{proof}

\begin{proposition}\label{P:minimal_vw}
Let $w\in L_n$ with $|w|\geq 1$.  If $w=v\sigma_j$, then $F_n^{\min}(w)$ is the set of $\preccurlyeq$-minimal elements in
$$
   \{\sigma_1,\ldots,\sigma_{j-2},\sigma_{j-1}\sigma_j\}
      \cup \{\sigma_j\backslash \beta \mid \beta\in F_n^{\min}(v)\}.
$$
\end{proposition}

\begin{proof}
The case $|w|=1$ is clear from Proposition~\ref{P:minimal_generator}, as $F_n^{\min}(\epsilon)=\emptyset$. Suppose that $|w|>1$ and let $\alpha$ be a forbidden prefix after $w=v\sigma_j$, that is,
$v\sigma_j \omega(\alpha)\notin L_n$.
We will distinguish two cases: either $\sigma_j \omega(\alpha)$ is a lex-representative, or not.

If $\sigma_j \omega(\alpha)$ is not a lex-representative, then $\alpha$ is a forbidden prefix after $\sigma_j$.  By Proposition~\ref{P:minimal_generator}, in this case $\alpha$ admits a prefix $\alpha'\in\{\sigma_1,\ldots,\sigma_{j-2},\sigma_{j-1}\sigma_j\}$.

If $\sigma_j \omega(\alpha)$ is a lex-representative, then $\sigma_j \alpha$ is a forbidden prefix after $v$, hence there is $\beta\in F_n^{\min}(v)$ such that $\beta\preccurlyeq \sigma_j\alpha$, that is, $\sigma_j\backslash \beta \preccurlyeq \alpha$.

We have then shown that if $\alpha$ is a forbidden prefix after $w$, then $\alpha$ admits a prefix
$\alpha'\in\{\sigma_1,\ldots,\sigma_{j-2},\sigma_{j-1}\sigma_j\}
   \cup \{\sigma_j\backslash \beta \mid \beta\in F_n^{\min}(v)\}$.
Since the elements of the latter set are forbidden prefixes after $w$ by construction, the claim follows.
\end{proof}

The following lemma will give us control over the elements of the form $\sigma_j\backslash \beta$.

\begin{lemma}\label{L:calculations}
For $i\in\{2,\ldots,n-1\}$ and $j\in\{1,\ldots,n-1\}$, one has:
$$
\sigma_j\backslash \sigma_{i-1}\sigma_i =
   \begin{cases}
     (\sigma_{i-1}\sigma_{i})(\sigma_{i-2}\sigma_{i-1}) & \mbox{ if $j=i-2$} \\
     \sigma_{i}                         & \mbox{ if $j=i-1$}   \\
     \sigma_{i-1}\sigma_{i}\sigma_{i+1} & \mbox{ if $j=i+1$}  \\
     \sigma_{i-1}\sigma_i               & \mbox{ otherwise}
    \end{cases}
$$
For $j\in\{1,\ldots,n-1\}$ and $1\leq m\leq i\leq n-1$, one has:
$$
\sigma_j\backslash \sigma_{i}\sigma_{i-1}\cdots \sigma_m =
   \begin{cases}
     \sigma_{i}\sigma_{i-1}\cdots \sigma_m      & \mbox{if $j\neq m-1, i, i+1$}  \\
     \sigma_{i}\sigma_{i-1}\cdots \sigma_{m-1}  & \mbox{if $j=m-1$}              \\
     \sigma_{i-1}\sigma_{i-2}\cdots \sigma_{m}  & \mbox{if $j=i$}                \\
     (\sigma_{i}\sigma_{i+1})(\sigma_{i-1}\sigma_i)\cdots (\sigma_{m} \sigma_{m+1})
                                                & \mbox{if $j= i+1$}
   \end{cases}
$$
\end{lemma}

\begin{proof}
For $x\in M$ let $\pi_x\in\Sym_n$ denote the permutation that $x$ induces on the $n$ strands.
A so-called \textsl{permutation braid} is a braid in which  any two strands cross at most once.
If $x$ is a permutation braid, then $x$ is uniquely determined by $\pi_x$.  Moreover,
one has $\sigma_i\preccurlyeq x$ if and only if $\pi_x(i)>\pi_x(i+1)$, that is, if and only if strands $i$ and $i+1$ cross in $x$.
In particular, the least common multiple of two permutation braids $x$ and $y$ can be computed easily from $\pi_x$ and $\pi_y$ \cite{Epstein}.

All the braids occurring in the statement of the lemma are permutation braids and the claimed equalities are readily checked.
\end{proof}

\begin{example}\label{E:forbidden_prefixes}
Consider $\sigma_4\sigma_3\sigma_2\sigma_2\sigma_1 \in B_5$.
From Proposition~\ref{P:minimal_generator} one has
$F_n^{\min}(\sigma_4) = \{ \sigma_1 , \sigma_2 , \sigma_3\sigma_4 \}$.
Using Proposition~\ref{P:minimal_vw} and Lemma~\ref{L:calculations} repeatedly, one then obtains:
\begin{align*}
& F_n^{\min}(\sigma_4\sigma_3)
 = \min_\preccurlyeq \big\{ \sigma_1 , \sigma_2\sigma_3 ,
     \sigma_3\backslash\sigma_1 , \sigma_3\backslash\sigma_2 , \sigma_3\backslash \sigma_3\sigma_4 \big\}
 = \{ \sigma_1 , \sigma_2\sigma_3 , \sigma_4 \} \\
& F_n^{\min}(\sigma_4\sigma_3\sigma_2)
 = \min_\preccurlyeq \big\{ \sigma_1\sigma_2 ,
     \sigma_2\backslash\sigma_1 , \sigma_2\backslash\sigma_2\sigma_3 , \sigma_2\backslash\sigma_4 \big\}
 = \{ \sigma_1\sigma_2 , \sigma_3 , \sigma_4 \} \\
& F_n^{\min}(\sigma_4\sigma_3\sigma_2\sigma_2)
 = \min_\preccurlyeq \big\{ \sigma_1\sigma_2 ,
     \sigma_2\backslash\sigma_1\sigma_2 , \sigma_2\backslash\sigma_3 , \sigma_2\backslash\sigma_4 \big\}
 = \{ \sigma_1\sigma_2 , \sigma_3\sigma_2 , \sigma_4\} \\
& F_n^{\min}(\sigma_4\sigma_3\sigma_2\sigma_2\sigma_1)
 = \min_\preccurlyeq \big\{ \sigma_1\backslash\sigma_1\sigma_2 , \sigma_1\backslash\sigma_3\sigma_2 ,
     \sigma_1\backslash\sigma_4 \big\}
 = \{ \sigma_2 , \sigma_3\sigma_2\sigma_1 , \sigma_4 \}
\end{align*}
\end{example}

\begin{definition}\label{D:forbidden_prefixes_encoding}
Let $[n]=\{1,\dots,n-1\}$.
\vspace*{-1ex}
\begin{enumerate}[\textup{(}a\textup{)}]
\item
A function $f:[n]\to\{-1,0\}\cup[n]$ is called {\bf admissible}, if it satisfies
$f(i)\le i$ for all $i\in[n]$, and $f(1)\neq -1$.

\item
For an admissible function $f$, we define
\[
   F_f =
    \big\{ \sigma_i\sigma_{i-1}\cdots\sigma_{f(i)} \;\big|\; i\in f^{-1}\big([n]\big) \big\}
    \cup \big\{ \sigma_{i-1}\sigma_i \;\big|\; i\in f^{-1}(-1) \big\}
    \subset M\,.
\]
\end{enumerate}
\end{definition}

We now show that for any $w\in L_n$, there is an admissible function $f_w$
such that $F_n^{\min}(w) = F_{f_w}$, and we express $f_{(w\sigma_j)}$, for $w\sigma_j\in L_n$, in terms of~$f_w$.

\begin{proposition}\label{P:minimal_generator_encoding}
\quad\vspace*{-1ex}
\begin{enumerate}[\textup{(}a\textup{)}]
\item\label{P:minimal_generator_encoding:a}
$F_n^{\min}(\epsilon) = F_f$, where $f:[n]\to\{-1,0\}\cup[n]$ given by $f(i)=0$ for all $i$ is admissible.
\item\label{P:minimal_generator_encoding:b}
$F_n^{\min}(\sigma_j) = F_{g_j}$, where $g_j:[n]\to\{-1,0\}\cup[n]$ given by
\[
 g_j(i) =
  \begin{cases}
    i  & \text{if $1\le i\le j-2$} \\
    -1 & \text{if $i=j>1$} \\
    0  & \text{otherwise}
  \end{cases}\vspace*{-1.5ex}
\]
is admissible.
\end{enumerate}
\end{proposition}
\begin{proof}
\claim{\ref{P:minimal_generator_encoding:a}} is trivial.  \claim{\ref{P:minimal_generator_encoding:b}} is just a restatement of Proposition~\ref{P:minimal_generator}.
\end{proof}

\begin{proposition}\label{P:minimal_vw_encoding}
Let $w\in L_n$ and let $f$ be admissible with $F_n^{\min}(w) = F_f$.
Then $w\sigma_j\in L_n$ if and only if $f(j)\ne j$.  Moreover, in this case, $F_n^{\min}(w\sigma_j) = F_g$, for the admissible function $g$ given as follows:
\begin{align*}
\text{for $i< j-1$:}\qquad   g(i) & = i \\
 g(j-1) & =\begin{cases}
  f(j) & \text{if $f(j)>0$} \\[0.5ex]
  0 & \text{otherwise}
\end{cases}  \\[1ex]
 g(j) & =\begin{cases}
 0 & \text{if $j=1$ or $f(j)=j-1$} \\[0.5ex]
 -1 & \text{otherwise}
\end{cases} \\[1ex]
\text{for $i>j$:}\qquad  g(i) & =\begin{cases}
i & \text{if $f(i)=-1$} \\[0.5ex]
 j & \text{if $f(i)=j+1$} \\[0.5ex]
 f(i) & \text{otherwise}
\end{cases} \\
\end{align*}
\end{proposition}
\begin{proof}
One has $w\sigma_j\in L_n$ if and only if $\sigma_j\notin F_n(w)$; the latter is the case if and only if $f(j)\ne j$, as $\sigma_j$ is an atom.  Now assume $w\sigma_j\in L_n$, and thus $f(j)\ne j$. Then it is clear that $g$ is admissible, as $f(j)\leq j-1$.

Recall from Proposition~\ref{P:minimal_vw} that $F_n^{\min}(w\sigma_j)$  is the set of minimal elements in $\{\sigma_1,\ldots,\sigma_{j-2},\sigma_{j-1}\sigma_j\}
\cup \{\sigma_j\backslash \beta \mid \beta\in F_n^{\min}(w)\}.$ This implies that $\sigma_1,\ldots,\sigma_{j-2}$ belong to $F_n^{\min}(w\sigma_j)$, as they are atoms, so $g(i)=i$ for $i<j-1$. Also, $F_n^{\min}(w\sigma_j)$ contains $\sigma_{j-1}\sigma_j$, unless $j=1$ or $\sigma_{j-1}\in F_n^{\min}(w\sigma_j)$.
We have $\sigma_{j-1}\in F_n^{\min}(w\sigma_j)$ if and only if $\sigma_{j}\backslash \beta =\sigma_{j-1}$ for some $\beta\in F_n^{\min}(w)$.  By Lemma~\ref{L:calculations}, the latter is equivalent to $\beta=\sigma_{j}\sigma_{j-1}$, and hence to $f(j)=j-1$.

It follows by induction on $|w|$ that one can have $f(i)=-1$ only if $\sigma_i$ is the last letter of $w$. In that case, $\sigma_1,\ldots,\sigma_{i-2}$ are forbidden prefixes after $w$, hence $j\geq i-1$. Therefore, if $i>j$ we can have $f(i)=-1$ only if $i=j+1$. As $\sigma_j\backslash \sigma_j\sigma_{j+1}=\sigma_{j+1}$, this is why in this case we have $g(i)=i$.

The other values of $g$ follow directly from Proposition~\ref{P:minimal_vw}, using the identities from Lemma~\ref{L:calculations} and discarding any elements which are not minimal (specifically, any multiples of $\sigma_1,\dots,\sigma_{j-2}$ and $\sigma_{j-1}\sigma_j$).
\end{proof}

\begin{corollary}\label{C:existence_encoding}
If $w\in L_n$, then there exists a unique admissible function $f_w$, such that $F_n^{\min}(w) = F_{f_w}$.
In particular, $|F_n^{\min}(w)| \le n-1$.
\end{corollary}
\begin{proof}
If it exists, $f_w$ is uniquely determined by $F_n^{\min}(w)$ and thus by $w$.
The existence follows by induction on $|w|$, using Propositions~\ref{P:minimal_generator_encoding} and \ref{P:minimal_vw_encoding}.
\end{proof}

\begin{example}\label{E:forbidden_prefixes_encoding}
Consider again $\sigma_4\sigma_3\sigma_2\sigma_2\sigma_1 \in B_5$.
To shorten notation, we identify an admissible function $f$ with the sequence $[f(1),\dots,f(n-1)]$.
From Proposition~\ref{P:minimal_generator_encoding} one has
$f_{\sigma_4} = [ 1, 2, 0, -1 ]$.
Repeated application of Proposition~\ref{P:minimal_vw_encoding} then yields
$f_{\sigma_4\sigma_3} = [ 1, 0, -1, 4 ]$,
$f_{\sigma_4\sigma_3\sigma_2} = [ 0, -1, 3, 4 ]$,
$f_{\sigma_4\sigma_3\sigma_2\sigma_2} = [ 0, -1, 2, 4 ]$, and
$f_{\sigma_4\sigma_3\sigma_2\sigma_2\sigma_1} = [ 0, 2, 1, 4 ]$.
Observe that the sets of minimal forbidden prefixes described by these functions are exactly those computed
in Example~\ref{E:forbidden_prefixes}.
\end{example}

Given $w\in L_n$, the set $F_n^{\min}(w)$ can be computed efficiently using Propositions~\ref{P:minimal_generator_encoding} and \ref{P:minimal_vw_encoding}; the time required is obviously linear in $n$ and $|w|$.
\medskip

There are many admissible functions, but the ones corresponding to minimal sets of forbidden prefixes are very special: It can be shown that they are exactly those satisfying the conditions in the following corollary.
However, proving that these conditions are sufficient is quite technical; as we do not use this fact in the sequel, we only show that they are necessary.

\begin{corollary}\label{C:properties_encoding}
Let $w=w'\sigma_j\in L_n$. Let $i_1<i_2<\cdots<i_r$ be all the indices greater than $j$ such that $0<f_w(i_t)<i_t$ ($t=1,\dots,r$). Denote $m=f_w(j-1)$.%
\vspace*{-1ex}
\begin{enumerate}[\textup{(}a\textup{)}]

\item\label{C:properties_encoding:a}
$f_w(i)=i$ for $i<j-1$.

\item\label{C:properties_encoding:b}
$f_w(i)\in\{0,\ldots, i\}$ for all $i\neq j$.

\item\label{C:properties_encoding:c}
$f_w(j)=\begin{cases}
          0 & \text{if either $j=1$ or $f_w(j-1)=j-1$} \\
          -1 & \text{otherwise}
        \end{cases}$

\item\label{C:properties_encoding:d}
If $f_w(i)=0$ for some $i>j$, then $f_w(\ell)=0$ for $\ell=i,\ldots,n-1$.

\item\label{C:properties_encoding:e}
If $r\ge 1$, then $0<f_w(i_r)\le \cdots \le f_w(i_1)\le j$.

\item\label{C:properties_encoding:f}
If $r\ge 1$ and $m>0$, then either $f_w(i_1)=j$ or $f_w(i_1)\leq m$.

\item\label{C:properties_encoding:g}
If $r\ge 2$ and $m>0$, then $0<f_w(i_r)\le \cdots \le f_w(i_2)\le m$.

\end{enumerate}
\end{corollary}

\begin{proof}
If $w'=\epsilon$, that is $w=\sigma_j$, the results holds trivially by Proposition~\ref{P:minimal_generator_encoding}, so let $w'=w''\sigma_k$ and assume the result for $f_{w'}$.
To shorten notation, let $f=f_{w'}$ and $g=f_w$.
Since $w'\sigma_j\in L_n$, we have $f(j)\ne j$ and thus $j\ge k-1$.
Proposition~\ref{P:minimal_vw_encoding} immediately yields \claims{\ref{C:properties_encoding:a}}, \noclaim{\ref{C:properties_encoding:b}} and \noclaim{\ref{C:properties_encoding:c}}.

For \claim{\ref{C:properties_encoding:d}} assume that $g(i)=0$ for some $i>j$. By Proposition~\ref{P:minimal_vw_encoding} this happens only if $g(i)=f(i)=0$, and if $i>k$ this implies that $g(r)=f(r)=0$ for $r=i,\ldots,n-1$. As $i>j\ge k-1$, the only remaining case is $i=k=j+1$ and $f(k)=0$. In this case \claim{\ref{C:properties_encoding:c}} applied to $f$ gives $f(k-1)=k-1$, that is $f(j)=j$, which is a contradiction. Thus \claim{\ref{C:properties_encoding:d}} holds.

Now suppose $r\ge1$. We have $i_1>j\ge k-1$, and $i_1=k$ would imply $f(i_1)\in \{0,-1\}$ and then, with Proposition~\ref{P:minimal_vw_encoding}, $g(i_1)\in \{0,i_1\}$; the latter is a contradiction, so $k<i_1<\cdots <i_r$.
For $t=1,\ldots,r$, Proposition~\ref{P:minimal_vw_encoding} also yields that
either $g(i_t)=f(i_t)$, or $f(i_t)=j+1$ and $g(i_t)=j$.
Since $f(i_t)\le i_t$ and $j<i_1<i_2<\dots<i_r$, the latter implies that either $0<f(i_t)<i_t$, or $t=1$ and $i_1=j+1=f(i_1)=g(i_1)+1$.
In any case, applying \claim{\ref{C:properties_encoding:e}} to $f$ and using $k\le j+1$, we obtain $0< f(i_r)\leq \cdots \leq f(i_1)\leq j+1$.
Again using Proposition~\ref{P:minimal_vw_encoding} then yields $0< g(i_r)\leq \cdots \leq g(i_1)\leq j$, showing \claim{\ref{C:properties_encoding:e}}.

Suppose $m=g(j-1)>0$.  If $i_1=j+1=f(i_1)=g(i_1)+1$, then \claim{\ref{C:properties_encoding:f}} holds.
By the preceding paragraph, we may thus assume that $0<f(i_t)<i_t$ for $1\le t\le r$.
By Proposition~\ref{P:minimal_vw_encoding}, $m>0$ implies $f(j)>0$ and $m=f(j)$.
In particular, $j\ne k$ (as $f(k)\in \{0,-1\}$), whence we have either $j=k-1$ or $j>k$.
In the former case, $m=f(k-1)$, so we can apply the result to $f$ and obtain that $f(i_t)\leq m$ for $2\le t\le r$, and either $f(i_1)=k=j+1$ or $f(i_1)\leq m$;
as $m\le j+1$, Proposition~\ref{P:minimal_vw_encoding} then yields $g(i_t)=f(i_t)\leq m$ for $2\le t\le r$, and either $g(i_1)=j$ or $g(i_1)=f(i_1)\leq m$, showing \claims{\ref{C:properties_encoding:f}} and \noclaim{\ref{C:properties_encoding:g}} in this case.
On the other hand, if $j>k$ we just need to notice that $0<m=f(j)<j$ and apply \claim{\ref{C:properties_encoding:e}} to $f$,
which yields $f(i_t)\leq f(j)=m< j$, whence $g(i_t)=f(i_t)\leq m$. So \claims{\ref{C:properties_encoding:f}} and \noclaim{\ref{C:properties_encoding:g}} hold.
\end{proof}

One can show that the conditions of Corollary~\ref{C:properties_encoding} are sufficient for $f$ to be the defining function of some $F_n^{\min}(w)$ by constructing, for every $f$ satisfying the conditions, a word $w_f$ such that $F_n^{\min}(w_f)=F_f$.
As this construction for general $f$ is very technical, we do not describe it here.  Instead, we only consider some special cases; enough to show that the number of possible sets $F_n^{\min}(w)$ is exponential in $n$.

\begin{proposition}\label{P:images_encoding}
For $1\le i,j\le n-1$ define $[i,j]\in \A^*$ as the product $\sigma_i\sigma_{i\pm1}\cdots\sigma_j$.
To shorten notation, we identify an admissible function $f$ with the sequence $\big[f(1),\dots,f(n-1)\big]$.
\begin{enumerate}[\textup{(}a\textup{)}]\vspace*{-1ex}
\item One has $[n-1,1]\in L_n$ and
  \[
    f_{[n-1,j]} =
      \begin{cases}
        \big[1,\dots,j-2,0,-1,j+1,\dots,n-1\big]  & \text{if $n-1\ge j>1$} \\[1ex]
        \big[0,2,\dots,n-1\big]                   & \text{if $j=1$}
      \end{cases}
  \]
\item If $i\in[n]$ and $w\in L_n$ such that $f_w = \big[0,2,3,\dots,i,m_{i+1},\dots,m_{n-1}\big]$,
  with $m_k\in \{1,k\}$ for $k=i+1,\dots,n-1$, then $w\cdot[1,i-1]\in L_n$ and
  \[
    f_{w\cdot[1,j]} =
      \begin{cases}
        \big[1,\dots,j-1,0,j,j+2,\dots,i,m_{i+1},\dots,m_{n-1}\big]  & \!\!\text{if $1\le j<i-1$} \\[1ex]
        \big[1,\dots,i-2,0,i-1,m_{i+1},\dots,m_{n-1}\big]            & \!\!\text{if $j=i-1$}
      \end{cases}
  \]
\item If $i\in[n]$ and $w\in L_n$ such that $f_w = \big[1,\dots,i-2,0,i-1,m_{i+1},\dots,m_{n-1}\big]$,
  with $m_k\in \{1,k\}$ for $k=i+1,\dots,n-1$, then $w\cdot[i-1,1]\in L_n$ and
  \[
    f_{w\cdot[i-1,j]} =
      \begin{cases}
        \big[1,\dots,j-2,0,-1,j+1,\dots,i-1,j,m_{i+1},\dots,m_{n-1}\big] \hspace*{-2cm} \\[1ex]
                     & \text{if $i-1\ge j>1$} \\[1ex]
        \big[0,2,3,\dots,i-1,1,m_{i+1},\dots,m_{n-1}\big]                 & \text{if $j=1$}
      \end{cases}
  \]
\end{enumerate}
\end{proposition}
\begin{proof}
The claims easily follow by induction on $j$, using Proposition~\ref{P:minimal_vw_encoding}.
\end{proof}

\begin{corollary}\label{C:images_encoding}
If $S=\{i_1,\dots,i_r\}$ is a subset of $\{1,\dots,n-2\}$, possibly empty, with $i_1>\dots>i_r$,
then $w_S = [n-1,1]\cdot[1,i_1]\cdot[i_1,1]\cdots\cdot[1,i_r]\cdot[i_r,1]\in L_n$, and
\[
   f_{w_S}(j) =
     \begin{cases}
       0 & \text{if $j=1$} \\
       1 & \text{if $j-1\in S$} \\
       j & \text{otherwise} \;.
     \end{cases}
\]
\end{corollary}
\begin{proof}
The claim follows from Proposition~\ref{P:images_encoding} by induction on $r$.
\end{proof}

\begin{corollary}\label{C:lower_bound}
Let $n>1$. The set $\{F_n^{\min}(w) \mid w\in L_n\}$ is finite, but it has at least $2^{n-2}$ elements.
\end{corollary}

\begin{proof}
The set is finite as each $F_n^{\min}(w)$ is equal to $F_f$ for some admissible function $f$. It has at least $2^{n-2}$ elements by Corollary~\ref{C:images_encoding}
\end{proof}

Before using, in Section~\ref{S:counting_braids_with_prefixes}, our description of forbidden prefixes to compute the numbers $x_{n,k}(w,m)$ required by Algorithm~\ref{A:random}, we show in the next section that minimal forbidden prefixes yield a finite state automaton accepting the language of lex-representative words, that has the minimal possible number of states.

\section{A minimal finite state automaton accepting~$L_n$}\label{S:automaton}

The sets of minimal forbidden prefixes after a given word $w\in L_n$ provide a very natural way to construct a finite state automaton that accepts the language $L_n$ of lex-representatives of braids in $M$. Indeed, we will see that this finite state automaton is minimal, in the sense that it has the minimal possible number of states.

Recall that a finite state automaton is a quintuple $\Gamma=(\SS,\A,\mu,Y,S_0)$ where $\SS$ is a finite set, $\A$ is the alphabet, $\mu:\ \SS\times \A \rightarrow \SS$ is the transition function, $Y\subseteq \SS$ is the set of accepted states and $S_0\in \SS$ is the initial state~\cite{Epstein}.
We extend $\mu$ to a function $\SS\times \A^* \rightarrow \SS$, also denoted by $\mu$, in the natural way.

\begin{definition}
Denote $\Gamma_n=(\SS,\A,\mu,Y,S_0)$, where
$\A = \{\sigma_1,\ldots,\sigma_{n-1}\}$,\linebreak
$\SS = \big\{S\subset \A^* \mid S=F_n^{\min}(w) \mbox{ for some } w\in L_n\big\}\cup \{\A\}$,
$S_0 = \emptyset = F_n^{\min}(\epsilon)$,
$Y = \SS\backslash \{\A\}$,
and $\mu:\SS\times\A\to\SS$ is given by
\[
  (S,\sigma_i) \mapsto
    \begin{cases}
       \A & \text{if $\sigma_i\in S$} \\
       \displaystyle\min_{\preccurlyeq}\big(\{\sigma_1,\ldots,\sigma_{i-2},\sigma_{i-1}\sigma_i\}
                   \cup \{\sigma_i\backslash \beta\mid\beta\in S\}\big) & \text{otherwise} \;.
    \end{cases}
\]
\end{definition}

\begin{proposition}\label{P_Gamma_minimal}
$\Gamma_n$ is a finite state automaton accepting the language $L_n$.
Moreover, any finite state automaton accepting the language $L_n$ has at least as many states as $\Gamma_n$.
\end{proposition}

\begin{proof}
By Corollary~\ref{C:existence_encoding}, every state $S=F_n^{\min}(w)\in\SS$ is uniquely determined by an admissible function.  As the number of admissible functions is finite, $\Gamma_n$ is indeed a finite state automaton.

If $w\in L_n$, then it follows from Proposition~\ref{P:minimal_vw} by induction on $|w|$ that
$\mu(S_0,w)=F_n^{\min}(w)\in Y$, so $w$ is accepted.
Conversely, if $w\notin L_n$, let $w=v\sigma_i v'$, where $v$ is the initial subword of maximal length of $w$ that lies in $L_n$.  As above, we have $\mu(S_0,v)=F_n^{\min}(v)$.  Moreover, since $v\sigma_i\notin L_n$, we have $\sigma_i\in F_n(v)$ and thus $\sigma_i\in F_n^{\min}(v)=\mu(S_0,v)$, as $\sigma_i$ is an atom.
Then, by the definition of $\mu$, we have $\mu(S_0,v\sigma_i)=\A$, and further (by induction on $|v'|$)
$\mu(S_0,w)=\mu(S_0,v\sigma_i v')=\A$, so $w$ is not accepted.

In order to show that $\Gamma_n$ has the minimal possible number of states, assume that $\Gamma'$ is a finite state automaton accepting the language $L_n$.
Let $w_1, w_2\in L_n$ be such that $F_n^{\min}(w_1)\neq F_n^{\min}(w_2)$.
This implies $F_n(w_1)\neq F_n(w_2)$, so by symmetry, assume that there exists some
$\alpha\in F_n(w_1)\backslash F_n(w_2)$, and consider the word $w=\omega(\alpha)\in L_n$. As $\alpha$ is forbidden after $w_1$, we have $w_1w\notin L_n$, so reading $w$ starting in the state corresponding to $w_1$ one ends at a fail state.
However, as $\alpha$ is not forbidden after $w_2$, one has $w_2w\in L_n$, so reading $w$ starting in the state corresponding to $w_2$ one ends at an accepted state; in particular, the states corresponding to $w_1$ and $w_2$ in $\Gamma'$ must be distinct.
As $\Gamma'$ has at least one fail state, the number of states of $\Gamma'$ is at least the number of states of~$\Gamma_n$.
\end{proof}

\begin{example} For $n=3$, the automaton $\Gamma_n$ has 5 accepted states; it is represented in Figure~\ref{FSA_B3}. Recall that the initial state corresponds to the empty set, and notice that we did not represent the fail state:  as usual, arrows which are not drawn lead to the fail state $\{\sigma_1,\sigma_2\}$.
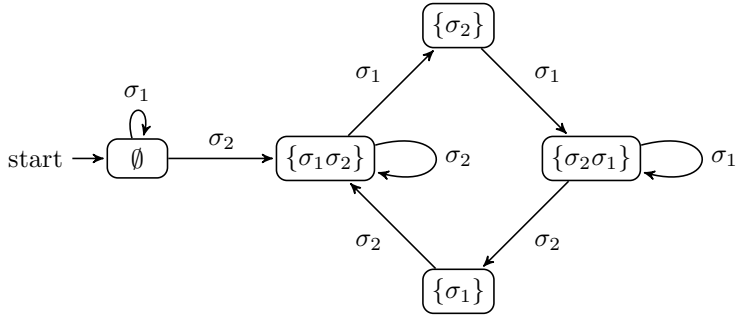
\begin{figure}[ht]
\begin{center}
\begin{tikzpicture}[->,>=stealth',shorten >=1pt,auto,node distance=2.5cm, semithick]
  \node[initial,rounded corners,draw] (1)                    {$\;\;\emptyset\;\;$};
  \node[rounded corners,draw]         (2) [right of=1]       {$\{ \sigma_1\sigma_2 \}$};
  \node[rounded corners,draw]         (3) [above right of=2] {$\{ \sigma_2 \}$};
  \node[rounded corners,draw]         (4) [below right of=3] {$\{ \sigma_2\sigma_1 \}$};
  \node[rounded corners,draw]         (5) [below right of=2] {$\{ \sigma_1 \}$};

  \path (1) edge [loop above] node {$\sigma_1$} (1);
  \path (1) edge              node {$\sigma_2$} (2);
  \path (2) edge              node {$\sigma_1$} (3);
  \path (2) edge [loop right] node {$\sigma_2$} (2);
  \path (3) edge              node {$\sigma_1$} (4);
  \path (4) edge [loop right] node {$\sigma_1$} (4);
  \path (4) edge              node {$\sigma_2$} (5);
  \path (5) edge              node {$\sigma_2$} (2);
\end{tikzpicture}
\end{center}\vspace{-2ex}
\caption{Finite state automaton accepting lex-representatives in $B_3^+$.}\label{FSA_B3}
\end{figure}
\end{example}

\begin{example} For $n=4$, the automaton $\Gamma_n$ has 18 accepted states; it is represented in Figure~\ref{FSA_B4}.
As above, the initial state corresponds to the empty set, and we did not represent the fail state $\{\sigma_1,\sigma_2,\sigma_3\}$.

\begin{figure}[ht]
\begin{center}\tiny
\begin{tikzpicture}[->,>=stealth',shorten >=1pt,auto, semithick]
\node[initial,rounded corners,draw] (1) at (0,0) {$\;\;\emptyset\;\;$};
\node[rounded corners,draw] (2)  at (-2,-1)    {$\{ \sigma_1\sigma_2 \}$};
\node[rounded corners,draw] (3)  at (0,-2)     {$\{ \sigma_1,\sigma_2\sigma_3 \}$};
\node[rounded corners,draw] (4)  at (-2,-2)    {$\{ \sigma_2 \}$};
\node[rounded corners,draw] (5)  at (0,-5)     {$\{ \sigma_1\sigma_2,\sigma_3 \}$};
\node[rounded corners,draw] (6)  at (-2,-3)    {$\{ \sigma_2\sigma_1 \}$};
\node[rounded corners,draw] (7)  at (1,-4)     {$\{ \sigma_2,\sigma_3 \}$};
\node[rounded corners,draw] (8)  at (5.5,-4)   {$\{ \sigma_1\sigma_2,\sigma_3\sigma_2 \}$};
\node[rounded corners,draw] (9)  at (-2,-4)    {$\{ \sigma_1 \}$};
\node[rounded corners,draw] (10) at (2,-5)     {$\{ \sigma_2\sigma_1,\sigma_3 \}$};
\node[rounded corners,draw] (11) at (6,-2)     {$\{ \sigma_2,\sigma_3\sigma_2\sigma_1 \}$};
\node[rounded corners,draw] (12) at (1,-3)     {$\{ \sigma_1,\sigma_2 \}$};
\node[rounded corners,draw] (13) at (3,-4)     {$\{ \sigma_1,\sigma_3\sigma_2 \}$};
\node[rounded corners,draw] (14) at (3,1)      {$\{ \sigma_2\sigma_1,\sigma_3\sigma_2\sigma_1 \}$};
\node[rounded corners,draw] (15) at (1.5,-1)   {$\{ \sigma_1,\sigma_2\sigma_1,\sigma_2\sigma_3 \}$};
\node[rounded corners,draw] (16) at (4,0)      {$\{ \sigma_1,\sigma_3\sigma_2\sigma_1 \}$};
\node[rounded corners,draw] (17) at (2.5,-2.5) {$\{ \sigma_1,\sigma_3 \}$};
\node[rounded corners,draw] (18) at (4.5,-1)   {$\{ \sigma_1\sigma_2,\sigma_3\sigma_2\sigma_1 \}$};

\path (1)  edge [loop above,left] node {$\sigma_1$} (1);
\path (1)  edge  node {$\sigma_2$} (2);
\path (1)  edge [left] node {$\sigma_3$} (3);
\path (2)  edge [left] node {$\sigma_1$} (4);
\path (2)  edge [loop above] node {$\sigma_2$} (2);
\path (2)  edge [above] node {$\sigma_3$} (3);
\path (3)  edge [left] node {$\sigma_2$} (5);
\path (3)  edge [loop right,left] node {$\sigma_3$} (3);
\path (4)  edge [left] node {$\sigma_1$} (6);
\path (4)  edge [above] node {$\sigma_3$} (3);
\path (5)  edge [right] node {$\sigma_1$} (7);
\path (5)  edge [bend right=45,below right] node {$\sigma_2$} (8);
\path (6)  edge [loop left,right] node {$\sigma_1$} (6);
\path (6)  edge [left] node {$\sigma_2$} (9);
\path (6)  edge [left] node {$\sigma_3$} (3);
\path (7)  edge [right] node {$\sigma_1$} (10);
\path (8)  edge [right] node {$\sigma_1$} (11);
\path (8)  edge [loop right,left] node {$\sigma_2$} (8);
\path (8)  edge [above] node {$\sigma_3$} (12);
\path (9)  edge [bend left=90] node {$\sigma_2$} (2);
\path (9)  edge [right] node {$\sigma_3$} (3);
\path (10) edge [loop right, left] node {$\sigma_1$} (10);
\path (10) edge [left] node {$\sigma_2$} (13);
\path (11) edge [bend right=65,right] node {$\sigma_1$} (14);
\path (11) edge [below] node {$\sigma_3$} (15);
\path (12) edge [right] node {$\sigma_3$} (3);
\path (13) edge [below] node {$\sigma_2$} (8);
\path (13) edge [below left] node {$\sigma_3$} (12);
\path (14) edge [loop above,left] node {$\sigma_1$} (14);
\path (14) edge [below left] node {$\sigma_2$} (16);
\path (14) edge [left] node {$\sigma_3$} (15);
\path (15) edge [right] node {$\sigma_2$} (17);
\path (15) edge [left] node {$\sigma_3$} (3);
\path (16) edge [right] node {$\sigma_2$} (18);
\path (16) edge [above] node {$\sigma_3$} (15);
\path (17) edge [above] node {$\sigma_2$} (8);
\path (18) edge [left] node {$\sigma_1$} (11);
\path (18) edge [loop right,left] node {$\sigma_2$} (18);
\path (18) edge [above] node {$\sigma_3$} (15);
\end{tikzpicture}
\end{center}\vspace{-5ex}
\caption{Finite state automaton accepting lex-representatives in $B_4^+$.}\label{FSA_B4}
\end{figure}
\end{example}

\begin{proposition}\label{P:automaton_exponential}
Any finite state automaton accepting the language $L_n$ has at least $2^{n-2}$ states.
\end{proposition}
\begin{proof}
By Proposition~\ref{P_Gamma_minimal} it is sufficient to show that $\Gamma_n$ has at least $2^{n-2}$ states.
The latter follows from Corollary~\ref{C:lower_bound}.
\end{proof}

\begin{remark}\label{R:exact_number_of_states}\em
We mentioned that the admissible functions that correspond to sets of minimal forbidden prefixes are precisely those satisfying the conditions of Corollary~\ref{C:properties_encoding}. This allows to give an exact expression for the number of states of $\Gamma_n$. The arguments are extremely technical, however, and as we do not use this result in the sequel, we skip the details here.

Just to give an idea, the following table contains the number of accepted states of $\Gamma_n$, for $n=3,\ldots,20$.
\begin{center}
\footnotesize
\begin{tabular}{|c||c|c|c|c|c|c|c|c|c|c|c|}
\hline
n\rule[1.5ex]{0pt}{1ex} & 3 & 4 & 5 & 6 & 7 & 8 & 9 & 10 & 11 & 12 & 13 \\
\hline
$|Y|$ \rule[1.5ex]{0pt}{1ex}& 5 & 18 & 56 & 161 & 443 & 1190 & 3156 & 8315 & 21835 & 57246 & 149970\\
\hline
\end{tabular}\smallskip

\begin{tabular}{|c||c|c|c|c|c|c|c|}
\hline
n\rule[1.5ex]{0pt}{1ex} &  14 & 15 & 16 & 17 & 18 & 19 & 20 \\\hline
$|Y|$\rule[1.5ex]{0pt}{1ex} &  392743 & 1028351 & 2692416 & 7049018 & 18454775 & 48315461 & 126491780 \\
\hline
\end{tabular}
\end{center}
\end{remark}
\vskip0ex

\begin{remark}\label{R:FSA_methods}\em
There are generic methods to generate uniformly random words of a regular language,
such as the \textsl{recursive method} \cite{Wilf,Flajolet}.
These algorithms have a precomputation phase, in which a (minimal) acceptor for the language in question is computed.

The results of this section show that the number of states of a minimal acceptor for $L_n$ is exponential in~$n$.
Thus, generating uniformly random words of $L_n$ by generic language-theoretic methods has a time complexity respectively a space complexity that is exponential in $n$.
While efficient for small values of $n$, such approaches are not feasible for larger values of $n$, as the above figures show.
\end{remark}

\section{Counting braids with suitable prefixes}\label{S:counting_braids_with_prefixes}

In this section we finally give a method to compute $x_{n,k}(w,m)$, where $m\geq 1$ if $w=\epsilon$, and $m\geq j-1$ if $w=w'\sigma_j$.
We assume that we have already computed $h_{s,t}$ for $0\le s\le n$ and $0\le t\le k$, as well as $x_{n,s}$ for $0\le s\le k$.  (See Section~\ref{S:counting_braids}.)

Recall from Definition~\ref{D:x_n_k} that $x_{n,k}(w,m)$ is the number of lex-representatives of length $k$ of the form $wv$ where $v\in L_n$ does not start with $\sigma_i$ for $i\le m$.
As the relations of $M$ are homogeneous, this amounts to counting the
braids $x$ of length $k-|w|$ for which $w\omega(x)\in L_n$ and $\omega(x)$ does not start with $\sigma_i$ for $i\le m$.
Knowing the number of braids $x$ of length $k-|w|$, it suffices to count those braids $x$ of length $k-|w|$ for which $w\omega(x)\notin L_n$, or $\omega(x)$ starts with $\sigma_i$ for $i\le m$.
If $m=0$, the second condition is never satisfied and we simply need to count $F_n(w)\cap M_{k-|w|}$.
If $m>0$, however, we have to add $\sigma_1,\dots,\sigma_m$ to the set of forbidden prefixes.

\begin{definition}\label{D:forbidden_prefixes_m}
Let $w\in L_n$ and $m\in\{0,\ldots,n-1\}$. We define $F_n^{\min}(w,m)$ to be the set of minimal elements, with respect to $\preccurlyeq$, in $\{\sigma_1,\ldots,\sigma_m\}\cup F_n^{\min}(w)$.
\end{definition}

Notice that $F_n^{\min}(w,m)$ is obtained from $F_n^{\min}(w)$ by removing all elements that start with $\sigma_1,\ldots,\sigma_{m}$, and adding the elements $\sigma_1,\ldots,\sigma_{m}$.  If $m\geq j-1$, where $\sigma_j$ is the last letter of $w$, this simplifies the description of the set $F_{n}^{\min}(w,m)$:

\begin{proposition}\label{P:description_minimal_m_sets}
Let $w\in L_n$, let $j=1$ if $w=\epsilon$, or $w=w'\sigma_j$ otherwise, and let $m\ge j-1$.
There is an admissible function $f$ such that $F_n^{\min}(w,m)=F_f$.  Moreover, the following hold:
\begin{enumerate}[\textup{(}a\textup{)}]\vspace*{-1ex}
\item
$f(r)=r$ for $r=1,\dots,m$ and $f(r)\ge0$ for $r=m+1,\dots,n-1$.
\item
$f(j)=0$ if $m=j-1$ and $f(j)=j$ if $m\ge j$.
\item
If $i\in\{j+1,\dots,n-1\}$ and $f(i)=0$, then $f(r)=0$ for $r=i,\dots,n-1$.
\item
If $i\in[n]$ and $0<f(i)<i$, then $f(i)\le j$.
\item
If $i,i'\in[n]$, such that $0<f(i)<i$ and $0<f(i')<i'$, then $i'>i$ implies
$f(i)\ge f(i')$.
\end{enumerate}

In particular, all elements of $F_n^{\min}(w,m)$ are of the form $\sigma_r\sigma_{r-1}\cdots \sigma_{f(r)}$ with $f(r)\leq r$, and those elements that are longer than one letter are nested:
If $f(i_t)<i_t$ ($1\le t\le s$) for $j<i_1<\dots<i_s$ then $j\ge f(i_1)\ge f(i_2)\ge\dots\ge f(i_s)$.
\end{proposition}

\begin{proof}
This follows immediately from Corollary~\ref{C:properties_encoding} together with the fact that $\sigma_i\in F_n^{\min}(w,m)$ for $i=1,\dots,m$.
\end{proof}

We are now in a position to calculate $x_{n,k}(w,m)$.  For a non-empty finite set $S=\{a_1,\ldots,a_r\}\subset M$, we define $\bigvee S = a_1\vee \cdots \vee a_r$.  We also define $\bigvee\emptyset=1$.
Recall that $M_k=\{x\in M\mid |x|=k\}$; so $x_{n,k}=|M_k|$ and $x_{n,k}=0$ if $k<0$.

\begin{proposition}\label{P:x_expression}
Let $w\in L_n$ and $m\in \{1,\ldots,n-1\}$. One has
\[
    x_{n,k}(w,m)
     =\sum_{S\subseteq F_{n}^{\min}(w,m)}{(-1)^{|S|}x_{n,k-|w|-|\bigvee S|}}
     =\sum_{l=0}^{k-|w|} \big( x_{n,k-|w|-l} \,\cdot T_l(w,m) \big)
    \;,
\]
where
\[
  T_l(w,m) = \sum_{\substack{S\subseteq F_{n}^{\min}(w,m) \\[1ex] |\bigvee S|=l}}(-1)^{|S|} \;.
\]
\end{proposition}
\begin{proof}
The second equality is just a reordering of terms.
For $\alpha\in M$, let $(\alpha M)_k = \alpha M \cap M_k$.
As $M$ is a cancellative monoid with homogeneous relations
$(\alpha M)_k = \alpha M_{k-|\alpha|}=\{\alpha\beta \mid \beta\in M_{k-|\alpha|} \}$, that is,
$|(\alpha M)_k|=x_{n,k-|\alpha|}$.

The braids $x$ for which $w\omega(x)\notin L_n$ or $\omega(x)$ starts with $\sigma_i$ for $i\le m$ are exactly those that admit a prefix $\alpha\in F_{n}^{\min}(w,m)$.
Moreover, $|w\omega(x)|=|w|+|x|$.
Hence
\[
   x_{n,k}(w,m)  \;\;=\;\;
       x_{n,k-|w|} \;\;-\;\; \Bigg|\bigcup_{\alpha\in F_{n}^{\min}(w,m)}{(\alpha M)_{k-|w|}}\Bigg| \;.
\]
By the inclusion-exclusion principle,
\[
  \Bigg|\bigcup_{\alpha\in F_n^{\min}(w,m)}{(\alpha M)_{k-|w|}}\Bigg|
    \;\;=\;\; \sum_{\emptyset\neq S\subseteq F_n^{\min}(w,m)}(-1)^{|S|-1}
        \Bigg|\bigcap_{\alpha\in S}(\alpha M)_{k-|w|}\Bigg| \;.
\]
Noting that $\bigvee\emptyset=1\in M$, whence $|\bigvee\emptyset|=0$, the claim then follows from
\begin{align*}
\Bigg|\bigcap_{\alpha\in S}(\alpha M)_{k-|w|}\Bigg|
 &\;\;=\;\; \Bigg| \bigg(\bigcap_{\alpha\in S}\alpha M\bigg) \cap M_{k-|w|} \Bigg|
 \;\;=\;\; \Big| \big((\textstyle\bigvee S) M\big)\cap M_{k-|w|} \Big|  \\[1.5ex]
 &\;\;=\;\; \Big| \big(\textstyle\bigvee S\big) M_{k-|w|-|\bigvee S|} \Big|
 \;\;=\;\; x_{n,k-|w|-|\bigvee S|} \;.
\end{align*}
\vskip-4ex
\end{proof}

In order to obtain a sub-exponential algorithm, we need to compute $T_l(w,m)$ without explicitly summing over all subsets of $F_{n}^{\min}(w,m)$.
To this end, we use some special properties of the braid monoid.
By Proposition~\ref{P:description_minimal_m_sets}, the elements of $F_{n}^{\min}(w,m)$ are permutation braids, thus so is the least common multiple of any collection of them.
The crucial point is that, under certain conditions, we only need to know the final position of some strands in a permutation braid, in order to describe sufficiently its least common multiple with another permutation braid.
Specifically, we have the following lemma, which can be readily verified working with the induced permutations~\cite{Epstein}.

\begin{lemma}\label{L:vees}
Let $x\in \langle \sigma_{a},\sigma_{a+1},\ldots,\sigma_{b-1}\rangle^+ \subset M$ be a permutation braid and let $\pi_x$ be the permutation induced by $x$.  Let $\pi_x(a)=a+r$ and $\pi_x(b)=b-s$.
\begin{enumerate}[\textup{(}a\textup{)}]\vspace*{-1ex}
\item
If $a>1$, then $x'=\sigma_{a-1}\vee x = x (\sigma_{a-1}\sigma_a\cdots \sigma_{a-1+r})$.
Moreover,\vspace{-1ex}
\[
  \pi_{x'}(a-1) = a+r
  \qquad\text{and}\qquad
  \pi_{x'}(b) =
    \begin{cases}
      b-s   & \text{if $a+r<b-s$} \\
      b-s-1 & \text{if $a+r>b-s$}
    \end{cases}  \;.
  \vspace{-1ex}
\]
\item
If $b<n$, then $x'=\sigma_{b}\vee x = x (\sigma_{b}\sigma_{b-1}\cdots \sigma_{b-s})$.
Moreover,\vspace{-1ex}
\[
  \pi_{x'}(a) =
    \begin{cases}
      a+r   & \text{if $a+r<b-s$} \\
      a+r+1 & \text{if $a+r>b-s$}
    \end{cases}\vspace{-1ex}
  \qquad\text{and}\qquad
  \pi_{x'}(b+1)=b-s
  \;.
\]
\item
If $b<n$ then $x'=(\sigma_{b}\sigma_{b-1}\cdots \sigma_a) \vee x
                 = x (\sigma_{b}\sigma_{b-1}\cdots \sigma_{a})$.
Moreover,\vspace{-1ex}
\[
  \pi_{x'}(a)=a+r+1
  \qquad\text{and}\qquad
  \pi_{x'}(b+1)=a
  \;.
\]
\end{enumerate}
\end{lemma}

\subsubsection*{Evaluating $T_l(w,m)$}

Let $w\in L_n$, let $j=1$ if $w=\epsilon$, or $w=w'\sigma_j$ otherwise, and let $m\ge j-1$.

We work with a 3-dimensional array $T$ of size $(k-|w|+1)\times n \times n$, which is transformed in at most $n-1$ steps.
At each step, we have two distinguished positions $a$ and $b$ with $1\leq a \le b \leq n$, defining
$S_{a,b} = F_n^{\min}(w,m)\cap \langle \sigma_a,\sigma_{a+1},\ldots,\sigma_{b-1}\rangle^+$, and
for $l\in \{0,\ldots,k-|w|\}$ and $r,s\in \{0,\ldots,n-1\}$, the entry $T_{l,r,s}$ of $T$ is
\[
   T_{l,r,s} \;\; = \sum_{\substack{S\subseteq S_{a,b}\,,\;
                             |\bigvee S|=l \\[1ex]
                              \pi_{(\bigvee S)}(a)=a+r\,,\;
                              \pi_{(\bigvee S)}(b)=b-s
                              }}(-1)^{|S|} \;.
\]
Initially, we take $a=b=j$, so $S_{a,b}=\emptyset$.  Obviously, in this case all entries of $T$ are $0$, except $T_{0,0,0}=1$, which corresponds to $\bigvee\emptyset=1$.

At the final step, we will have $a=1$ and $b=n$, so $S_{a,b} = F_n^{\min}(w,m)$, whence
\[
   T_l(w,m) = \sum_{r=0}^{n-1}\sum_{s=0}^{n-1} T_{l,r,s} \;.
\]

In each step we replace $(a,b)$ by $(a',b')\in\{(a-1,b),(a,b+1)\}$:
If there is some $\sigma_{b}\sigma_{b-1}\cdots \sigma_s\in F_n^{\min}(w,m)$ with $s<a$ or if $b=n$ then we set $(a',b')=(a-1,b)$, otherwise we set $(a',b')=(a,b+1)$.
By Proposition~\ref{P:description_minimal_m_sets}, either $S_{a',b'}\backslash S_{a,b}=\emptyset$, or $S_{a',b'}\backslash S_{a,b}=\{\sigma_b\}$, or $S_{a',b'}\backslash S_{a,b}=\{\sigma_b\sigma_{b-1}\cdots \sigma_a\}$, or $S_{a',b'}\backslash S_{a,b}=\{\sigma_{a-1}\}$.
Hence, if $S_{a',b'}\backslash S_{a,b}=\{x\}\ne\emptyset$ and $y\in S_{a,b}$, then according to Lemma~\ref{L:vees}, the length, the displacement of strand $a'$, and the displacement of strand $b'$ of the element $x\vee y$ only depend on $x$, and the length, the displacement of strand $a$, and the displacement of strand $b$ of the element $y$. Therefore, the table $T$ can be updated when replacing $(a,b)$ by $(a',b')$:
The elements contributing to $T_{l,r,s}$ in the old table will in the updated table contribute to $T_{l,0,s}$ if $(a',b')=(a-1,b)$, respectively to $T_{l,r,0}$ if $(a',b')=(a,b+1)$.
If $S_{a',b'}\backslash S_{a,b}\ne\emptyset$, then in addition, their least common multiples with $x$ will contribute to the entry of the updated table given by Lemma~\ref{L:vees}.
\medskip

The arguments from this section yield Algorithm~\ref{A:countleaves} computing $x_{n,k}(w,m)$.

\begin{algorithm}[p]
\caption{Computing $x_{n,k}(w,m)$}
\label{A:countleaves}
\small\begin{algorithmic}[1]
\REQUIRE{Integers $n\ge2$ and $k\ge 0$.
  A word $w=\sigma_{a_1}\cdots\sigma_{a_t}\in \A_n^*$.
  An integer $m\le n-1$, where $m\geq 1$ if $w=\epsilon$ and $m\geq j-1$ if $w=v\sigma_j$.}
  The numbers $x_{n,0},\ldots, x_{n,k}$.
\ENSURE{The number $x_{n,k}(w,m)$.}\smallskip
\IF{$t>k$}
  \RETURN{0}
\ENDIF
\STATE{$F := \emptyset$}
\FOR{$i := 1$ to $t$}
  \IF{$\sigma_{a_i}\in F$}
    \RETURN{$0$}
  \ELSE
    \STATE{$F := \min_{\preccurlyeq}\left(
                  \{\sigma_{1},\ldots,\sigma_{a_i-2},\sigma_{a_i-1}\sigma_{a_i}\}
                  \cup \{\sigma_{a_i}\backslash \beta \mid \beta\in F\}\right)$}\label{A:countleaves:prefixes1}
  \ENDIF
\ENDFOR
\STATE{$F := \min_{\preccurlyeq}\left(  \{\sigma_{1},\ldots,\sigma_m\}\cup F \right)$}%
                                                                      \label{A:countleaves:prefixes2}
\STATE{$T_{l,r,s} := 0$ for all $(l,r,s)\in I=\{0,\dots,k-t\}\times\{0,\dots,n-1\}\times\{0,\dots,n-1\} $}%
                                                                             \label{A:countleaves:init}
\IF{$t=0$}
  \STATE{$a := 1$\,;\, $b := 1$\,;\, $\alpha := 0$\,;\, $T_{0,0,0} := 1$}
\ELSE
  \STATE{$a := a_t$\,;\, $b := a_t$\,;\, $\alpha := 0$\,;\, $T_{0,0,0} := 1$}
\ENDIF
\WHILE{$(a,b)\neq (1,n)$}
  \IF{$b=n$ or $\sigma_b\sigma_{b-1}\cdots \sigma_s\in F$ for some $s<a$}
    \STATE{$(a',b') := (a-1,b)$\,;\, $\alpha := \alpha+1$}
    \IF{$F\cap\langle \sigma_{a'},\ldots,\sigma_{b'-1}\rangle^+
         = F\cap\langle \sigma_a,\ldots,\sigma_{b-1}\rangle^+$}
      \STATE{\textbf{for} $(l,r,s)\in I$ \textbf{do} $ T'_{l,r,s} :=
                 \begin{cases}
                    \sum_{u=0}^{b-a}T_{l,u,s} & \text{if $r=0$} \\
                    0                         & \text{otherwise}
                 \end{cases}
             $}\label{A:countleaves:case1}
    \ELSIF{$(F\cap\langle \sigma_{a'},\ldots,\sigma_{b'-1}\rangle^+)
         \setminus (F\cap\langle \sigma_a,\ldots,\sigma_{b-1}\rangle^+)
         = \{ \sigma_{a-1} \}$}
      \STATE{\textbf{for} $(l,r,s)\in I$ \textbf{do} $ T'_{l,r,s} :=
                 \begin{cases}
                    \sum_{u=0}^{b-a}T_{l,u,s} & \text{if $r=0$} \\
                    -T_{l-r, r-1, s}          & \text{if $1\le r\le l$ and $r+s<\alpha$} \\
                    -T_{l-r, r-1, s-1}        & \text{if $1\le r\le l$ and $r+s>\alpha$} \\
                    0                         & \text{otherwise}
                 \end{cases}
             $}\label{A:countleaves:case2}
    \ENDIF
  \ELSE
    \STATE{$(a',b') := (a,b+1)$\,;\, $\alpha := \alpha+1$}
    \IF{$F\cap\langle \sigma_{a'},\ldots,\sigma_{b'-1}\rangle^+
         = F\cap\langle \sigma_a,\ldots,\sigma_{b-1}\rangle^+$}
      \STATE{\textbf{for} $(l,r,s)\in I$ \textbf{do} $ T'_{l,r,s} :=
                 \begin{cases}
                    \sum_{u=0}^{b-a}T_{l,r,u} & \text{if $s=0$} \\
                    0                         & \text{otherwise}
                 \end{cases}
             $}\label{A:countleaves:case3}
    \ELSIF{$(F\cap\langle \sigma_{a'},\ldots,\sigma_{b'-1}\rangle^+)
         \setminus (F\cap\langle \sigma_a,\ldots,\sigma_{b-1}\rangle^+)
         = \{ \sigma_b \}$}
      \STATE{\textbf{for} $(l,r,s)\in I$ \textbf{do} $ T'_{l,r,s} :=
                 \begin{cases}
                    \sum_{u=0}^{b-a}T_{l,r,u} & \text{if $s=0$} \\
                    -T_{l-s, r, s-1}          & \text{if $1\le s\le l$ and $r+s<\alpha$} \\
                    -T_{l-s, r-1, s-1}        & \text{if $1\le s\le l$ and $r+s>\alpha$} \\
                    0                         & \text{otherwise}
                 \end{cases}
             $}\label{A:countleaves:case4}
    \ELSIF{$(F\cap\langle \sigma_{a'},\ldots,\sigma_{b'-1}\rangle^+)
         \setminus (F\cap\langle \sigma_a,\ldots,\sigma_{b-1}\rangle^+)
         = \{ \sigma_b\cdots\sigma_a \}$}
      \STATE{\textbf{for} $(l,r,s)\in I$ \textbf{do} $ T'_{l,r,s} :=
                 \begin{cases}
                    \sum_{u=0}^{b-a}T_{l,r,u}           & \text{if $s=0$} \\
                    -\sum_{u=0}^{b-a}T_{l-\alpha,r-1,u} & \text{if $s=\alpha$, $l\ge \alpha$, $r\ge 1$} \\
                    0                                   & \text{otherwise}
                 \end{cases}
             $}\label{A:countleaves:case5}
    \ENDIF
  \ENDIF
  \STATE{$(a,b) := (a',b')$\,;\, $T := T'$}
\ENDWHILE\vspace{1ex}
\RETURN{$\sum_{l=0}^{k-t}\left(\sum_{r=0}^{n-1}\sum_{s=0}^{n-1}T_{l,r,s}\right) x_{n,k-t-l}$}%
                                                                        \label{A:countleaves:return}
\end{algorithmic}
\end{algorithm}

\section{Complexity and timing analysis}\label{S:Complexity}

The analysis of the worst-case complexity of Algorithm~\ref{A:random} is relatively straightforward.
We assume that the addition of two $N$-bit integers has cost $O(N)$, and that the multiplication of two $N$-bit integers, in the relevant range of $N$, has cost $O(N^{\alpha})$ with $\alpha = \log_2 3 \approx 1.585$ (Karatsuba multiplication).

\begin{proposition}\label{P:complexity}
\quad\vspace*{-1ex}
\begin{enumerate}[\textup{(}a\textup{)}]
\item\label{P:complexity:a}
The worst-case complexity of Algorithm~\ref{A:countleaves} is at most $O(n^4k^{\alpha+1})$.
\item\label{P:complexity:b}
The worst-case complexity of Algorithm~\ref{A:random} is at most $O(n^4\ln n\ k^{\alpha+2})$.
\end{enumerate}
\end{proposition}
\begin{proof}
The absolute values of all entries of the array $T$ are bounded by $2^n$, so the entries of $T$ have at most $n$ bits.
As it is known \cite[Theorem 10]{Vershik}
that the logarithmic volume $\lim_{k\to \infty}{\frac{\log x_{n,k}}{k}}$ of $B_n^+$ is bounded above by 4, the numbers $x_{n,j}$ for $0\le j\le k$ have $O(k)$ bits.

Consider Algorithm~\ref{A:countleaves} and recall that we can describe forbidden prefixes by admissible functions as explained in Section~\ref{S:Forbidden_prefixes}.
By Proposition~\ref{P:minimal_vw_encoding}, line~\ref{A:countleaves:prefixes1}, which is executed at most $k$ times, and line~\ref{A:countleaves:prefixes2} both have a cost of $O(n)$.
Line~\ref{A:countleaves:init} has cost $O(n^2k)$.
Now consider lines~\ref{A:countleaves:case1}, \ref{A:countleaves:case2}, \ref{A:countleaves:case3}, \ref{A:countleaves:case4} and \ref{A:countleaves:case5}: At most $n$ executions of these occur.
Each line involves $O(nk)$ sums of $O(n)$ terms, and possibly $O(n^2 k)$ assignments.  Since the operands have at most $O(n)$ bits, the total cost is at most $O(n^4 k)$.
Line~\ref{A:countleaves:return} involves $O(n^2 k)$ additions with operands of size $O(n)$, and
$k$ multiplications and additions with operands of size $O(n+k)$, so the cost is at most
$O(n^3 k + k(n+k)^\alpha) = O(n^3k^{\alpha+1})$.
All other lines have cost $O(1)$ and are executed at most $n$ times, so \claim{\ref{P:complexity:a}} is shown.

Consider Algorithm~\ref{A:random}.  The body of the for-loop (lines~\ref{A:random:for-start} to \ref{A:random:for-end}) is executed $k$ times, the body of the while-loop (lines~\ref{A:random:while-start} to \ref{A:random:while-end}) at most $O(\ln n)$ times.  The cost of the for-loop is dominated by the invocations of Algorithm~\ref{A:countleaves}; by \claim{\ref{P:complexity:a}}, the cost is at most
$O(n^4\ln n\ k^{\alpha+2})$.
Line~\ref{A:random:random} has a cost of $O(k)$.
From Section~\ref{S:counting_braids}, equation~(\ref{E:h_m_j}), calculating $h_{m,j}$ involves $m$ additions.  We need to calculate $h_{m,j}$ for $0\le m\le n$ and $0\le j\le k$, so overall $O(n^2k)$ additions are required; as it follows by induction on $n$ that the operands have at most $O(n\ln n)$ bits, this has a cost of at most $O(n^3\ln n\ k)$.  Similarly, from equation~(\ref{E:x_n_j}), the calculation of $x_{n,j}$ involves $O(j)$ additions and multiplications; we need to calculate $x_{n,j}$ for $0\le j\le k$, so overall this requires $O(k^2)$ additions and multiplications of operands with at most $O(n\ln n+k)$ bits, which has a cost of at most $O((n\ln n)^\alpha\ k^{\alpha+2})$.
This shows \claim{\ref{P:complexity:b}}.
\end{proof}

\begin{remark}\label{R:implementation}\em
There are several optimisations that can be applied when implementing Algorithms~\ref{A:random} and~\ref{A:countleaves}.  While they do not affect the estimate of the worst-case complexity, they have a significant impact on actual running times.
\begin{enumerate}\vspace*{-1ex}\addtolength{\itemsep}{-1ex}
\item\label{R:implementation:cache_prefixes}
It is clearly not necessary to compute the forbidden prefixes from scratch in every invocation of Algorithm~\ref{A:countleaves}.  If the forbidden prefixes of the generated part of the word are stored in Algorithm~\ref{A:random}, only one invocation of line~\ref{A:countleaves:prefixes1} of Algorithm~\ref{A:countleaves} is needed.
\item\label{R:implementation:Bronfman_init}
It is possible to initialise the array $T$ in Algorithm~\ref{A:countleaves} directly for the values of $a$ and $b$ corresponding to the largest range that does not involve a forbidden prefix consisting of more than one letter, using the data computed in Section~\ref{S:counting_braids}.
\item\label{R:implementation:T_sparse}
The array $T$ in Algorithm~\ref{A:countleaves} is in general very sparse, as its contents are often collapsed to the 2-dimensional subarrays given by $r=0$ respectively $s=0$.
Keeping track of the values of $r$ and $s$ that actually can contain non-zero entries and only considering those in any subsequent summations greatly reduces actual running times; see Remark~\ref{R:timings}.
\item\label{R:implementation:l_bounded}
Since the least common multiple of all forbidden prefixes is a permutation braid, its length is bounded above by $\binom{n}{2}$.
Similarly, $H_n(t)$ is a polynomial of degree $\binom{n}{2}$, so $h_{n,j}=0$ if $j>\binom{n}{2}$.
It is therefore possible to restrict some summations to $\min\{k,\binom{n}{2}\}$ terms.
It is clear from the proof of Proposition~\ref{P:complexity} that in this case (that is, $n$ constant and $k>\binom{n}{2}$) the complexity in terms of $k$ is reduced to $O(k^{\alpha+1})$.
\end{enumerate}
\end{remark}

Table~\ref{Tbl:timings} contains timing results for several values of $n$ and $k$, using the implementation of the algorithm in the C-kernel of the computational algebra package \Magma\ \cite{Magma} by the first author.
Computations were done with a development version of \textsc{Magma} V2.16 on a GNU\,/\,Linux system with an Intel E8400 64-bit CPU (core:~3\,GHz, FSB:~1333\,MHz) and a main memory bandwidth of 4.7\,GB/s (X38 chipset, dual channel DDR2 RAM, memory bus: 1066\,MHz).
The minimum sample size was 100.
The sample size was varied to achieve an approximate running time of 2.5 minutes per sample (where possible honouring the minimum sample size), and hence a sufficient degree of accuracy for all parameter values.

\begin{table}[ht]\begin{center}
\begin{tabular}{@{}r@{}r|cccccccc@{\,}}
         &       &                    \multicolumn{8}{c}{$n$}                            \\
         &       &      4 &      8 &     16 &     32 &     64 &    128 &    256 &    512 \\
  \hline\rule{0pt}{11pt}
         &    4  & 1.09e0 & 1.33e0 & 1.76e0 & 2.07e0 & 2.45e0 & 2.95e0 & 3.58e0 & 8.03e0 \\
         &    8  & 4.62e0 & 6.39e0 & 5.79e0 & 6.10e0 & 1.26e1 & 2.19e1 & 3.13e1 & 4.84e1 \\
         &   16  & 2.20e1 & 5.10e1 & 4.56e1 & 7.52e1 & 9.73e1 & 1.35e2 & 2.31e2 & 2.42e2 \\
         &   32  & 1.07e2 & 3.81e2 & 4.63e2 & 6.70e2 & 8.81e2 & 1.01e3 & 1.15e3 & 1.79e3 \\
    $k$  &   64  & 4.84e2 & 2.07e3 & 4.55e3 & 6.72e3 & 5.72e3 & 7.94e3 & 7.61e3 & 7.72e3 \\
         &  128  & 2.01e3 & 9.55e3 & 3.30e4 & 5.24e4 & 5.70e4 & 1.03e5 & 1.13e5 & 1.04e5 \\
         &  256  & 8.42e3 & 4.27e4 & 1.90e5 & 4.17e5 & 6.06e5 & 1.20e6 & 1.81e6 & 2.06e6 \\
         &  512  & 4.36e4 & 1.78e5 & 9.70e5 & 3.27e6 & 5.85e6 & 1.23e7 & 2.56e7 & 3.29e7 \\
         & 1024  & 2.52e5 & 7.24e5 & 4.77e6 & 2.07e7 & 5.22e7 & 1.25e8 & 2.78e8 & 4.29e8
\end{tabular}\end{center}\vspace*{-3ex}
\caption{Running times in $\mu s$ per random braid for different values of $n$ and $k$.}\label{Tbl:timings}
\end{table}

\begin{table}[ht]\begin{center}
\begin{tabular}{@{}@{}r|ccccccccc@{\,}}
         $k$   &    4 &    8 &   16 &   32 &   64 &  128 &  256 &  512 & 1024 \\ \hline\rule{0pt}{11pt}
         $e_n$ & 0.35 & 0.50 & 0.48 & 0.48 & 0.47 & 0.75 & 1.08 & 1.36 & 1.59
\end{tabular}\end{center}\vspace*{-3ex}
\caption{Observed exponent $e_n$ of $n$ for different values of $k$; see Remark~\ref{R:timings}.}\label{Tbl:exponent_n}
\end{table}

\begin{table}[ht!]\begin{center}
\begin{tabular}{@{}@{}r|cccccccc@{\,}}
                       $n$   &    4 &    8 &   16 &   32 &   64 &  128 &  256 &  512
                                                                       \\ \hline\rule{0pt}{11pt}
    $e_k$             & 2.30 & 2.41 & 2.80 & 3.02 & 3.08 & 3.19 & 3.27 & 3.21
\end{tabular}\end{center}\vspace*{-3ex}
\caption{Observed exponent $e_k$ of $k$ for different values of $n$; see Remark~\ref{R:timings}.}\label{Tbl:exponent_k}
\end{table}

\begin{remark}\label{R:timings}\em
Assuming that the observed running time $t$ is approximated by a relation of the form
$t = c\,n^{e_n}k^{e_k}$, the value of $e_n$ for a fixed value of $k$ can be obtained by applying regression analysis to the rows of Table~\ref{Tbl:timings}.  Similarly, the value of $e_k$ for a fixed value of $n$ can be obtained from the columns of Table~\ref{Tbl:timings}.
The observed values of $e_n$ and $e_k$ are given in Tables~\ref{Tbl:exponent_n} and \ref{Tbl:exponent_k}, respectively.
\begin{enumerate}\vspace{-1ex}\addtolength{\itemsep}{-1ex}
\item
The observed values of $e_k$ are close to $\alpha+2\approx 3.585$ for large values of~$n$.  For small values of $n$, where $k>\binom{n}{2}$ for most of the parameter range tested, the observed values of $e_k$ are closer to $\alpha+1\approx 2.585$.

This is exactly what should be expected from Proposition~\ref{P:complexity} and Remark~\ref{R:implementation}~(\ref{R:implementation:l_bounded}).
That is, as far as $k$ is concerned, the average case complexity is quite close to the worst-case complexity.
\item
The observed values of $e_n$ are significantly smaller than $4$, that is, the average case complexity in $n$ is significantly better than the worst-case bound of Proposition~\ref{P:complexity}.

This is not surprising in the light of Remark~\ref{R:implementation}~(\ref{R:implementation:T_sparse}):
Our implementation exploits the fact that the array $T$ is sparse and tends to contain non-zero entries only for relatively few values of the indices $r$ and $s$.  The experimental results suggest that the array $T$ effectively is far from 3-dimensional, especially if $k$ is relatively small.
We expect that the theoretical growth rate in $n$ will be achieved, if at all, only for extremely large values of $k$.
\end{enumerate}\vspace{-1ex}
\end{remark}




\bigskip
\noindent\textbf{Volker Gebhardt}\\
\noindent School of Computing, Engineering and Mathematics\\
University of Western Sydney,
Locked Bag 1797, Penrith NSW 2751, Australia.\\
\noindent E-mail: \texttt{v.gebhardt@uws.edu.au}

\bigskip
\noindent\textbf{J.\,Gonz\'alez-Meneses}\\
\noindent Departamento de \'{A}lgebra, Facultad de Matem\'{a}ticas, Instituto de Matem\'aticas (IMUS), Universidad de Sevilla,
Apdo.\,1160, 41080 Sevilla, Spain.\\
\noindent E-mail: \texttt{meneses@us.es}

\end{document}